\documentclass{amsart}
\usepackage{amsmath}
\usepackage{amsfonts}
\usepackage{amsthm}
\usepackage{color}
\usepackage{mathtools}
\usepackage{tikz-cd}
\usepackage{faktor}
\usepackage{amssymb}
\usepackage{graphicx}

\DeclareMathOperator{\Aut}{Aut}
\let\Im\relax
\DeclareMathOperator{\Im}{Image}

\newcommand{\homeo}{\textnormal{Homeo}^{+}}
\newcommand{\SuspT}{\Sigma_{T}X}
\renewcommand{\phi}{\varphi}

\newcommand{\Inf}{\textnormal{Inf} \hspace{.03in}}
\newcommand{\A}{\mathcal{A}}
\newcommand{\susxi}{\tilde{\xi}}
\newcommand{\widesim}[2][2]{
  \mathrel{\underset{#2}{\scalebox{#1}[1]{$\sim$}}}
}
\newcommand{\asy}{\textnormal{as}(T)}
\theoremstyle{plain} 
\newtheorem{theorem}[equation]{Theorem}
\newtheorem{lemma}[equation]{Lemma}
\newtheorem{proposition}[equation]{Proposition}

\newtheorem{corollary}[equation]{Corollary}

\theoremstyle{remark}

\newtheorem{remark}[equation]{Remark}
\newtheorem{definition}[equation]{Definition}

\newtheorem{example}[equation]{Example}
\input xy
\xyoption{all}
\numberwithin{equation}{section}
\title{The Mapping Class Group of a Minimal Subshift}
\author{Scott Schmieding \and
Kitty Yang}
\begin{document}
\begin{abstract}
For a homeomorphism $T \colon X \to X$ of a Cantor set $X$, the mapping class group $\mathcal{M}(T)$ is the group of isotopy classes of orientation-preserving self-homeomorphisms of the suspension $\Sigma_{T}X$. The group $\mathcal{M}(T)$ can be interpreted as the symmetry group of the system $(X,T)$ with respect to the flow equivalence relation.
We study $\mathcal{M}(T)$, focusing on the case when $(X,T)$ is a minimal subshift. We show that when $(X,T)$ is a subshift associated to a substitution, the group $\mathcal{M}(T)$ is an extension of $\mathbb{Z}$ by a finite group; for a large class of substitutions including Pisot type, this finite group is a quotient of the automorphism group of $(X,T)$. When $(X,T)$ is a minimal subshift of linear complexity satisfying a no-infinitesimals condition, we show that $\mathcal{M}(T)$ is virtually abelian. We also show that when $(X,T)$ is minimal, $\mathcal{M}(T)$ embeds into the Picard group of the crossed product algebra $C(X) \rtimes_{T} \mathbb{Z}$.
\end{abstract}
\maketitle
\tableofcontents
\section{Introduction}
By a Cantor system $(X,T)$ we mean a homeomorphism ${T \colon X \to X}$ where $X$ is a Cantor set; when $T$ is minimal we call $(X,T)$ a minimal Cantor system. The suspension $\Sigma_{T}X$ (sometimes called the mapping torus) comes with a natural $\mathbb{R}$-action whose orbits correspond to the $T$-orbits in $X$, and we let $\homeo \SuspT$ denote the topological group of homeomorphisms of $\Sigma_{T}X$ which respect the orientation of the flow. Maps $f,g \in \homeo \SuspT$ are isotopic if they are connected by a path in $\homeo \SuspT$, and we define the mapping class group $\mathcal{M}(T)$ to be the group of isotopy classes of maps in $\homeo \SuspT$.\\
\indent The purpose of this paper is to study the group $\mathcal{M}(T)$ when $(X,T)$ is a minimal Cantor system. Apart from some general analysis, we focus especially on the case when $(X,T)$ is conjugate to a subshift of low complexity. \\
\indent Cantor systems $(X,T), (Y,S)$ are said to be \emph{flow equivalent} if there is a homeomorphism $h \colon \Sigma_{T}X \to \Sigma_{S}Y$ such that $h$ preserves the orientation of the respective $\mathbb{R}$-orbits. While conjugate systems are flow equivalent, flow equivalence is typically a weaker relation than conjugacy; for example, any system $(X,T)$ is flow equivalent to the return action on a discrete cross section of $(X,T)$. For systems presented as the Vershik map on a simple ordered Bratteli diagram, flow equivalence corresponds to the relation generated by conjugacy and making a finite number of changes to the Bratteli diagram (see \cite[Theorem 3.8]{GPS1}). In general, flow equivalence is an important relation for studying Cantor systems and has applications to the study of various $C^*$-algebras associated to Cantor systems. In this framework, the group $\mathcal{M}(T)$ plays the role of a symmetry group of $(X,T)$ with respect to the flow equivalence relation.\\
\indent A fundamental study of $\mathcal{M}(T)$ when $(X,T)$ is an irreducible shift of finite type was undertaken in \cite{BCMCG}, which itself built on results from \cite{Boyleposfac} and \cite{ChuysurichayThesis}. Similar to \cite{BCMCG}, we are motivated in part by the connection between $\mathcal{M}(T)$ and the group of automorphisms $\Aut(T)$, which consists of all homeomorphisms of $X$ that commute with $T$. The group $\Aut(T)$ has been the subject of significant interest; when $(X,T)$ is a shift of finite type, $\Aut(T)$ is known to possess a rich structure (see \cite{BLR88}, \cite{KR90}, \cite{Yang2018}), while recent results (e.g. \cite{CyrKraStretchedExponential}, \cite{CyrKraSubquadratic}, \cite{CyrKraBeyondTransitivity}, \cite{DDMPlowcomp}, \cite{CFKPDistortion}) indicate $\Aut(T)$ is more tame when $(X,T)$ is a subshift of lower complexity. Toward understanding the relation between the groups $\Aut(T)$ and $\mathcal{M}(T)$, in Proposition \ref{prop:autembedding} we show that when $(X,T)$ is minimal, $\Aut(T) / \langle T \rangle$ always embeds into $\mathcal{M}(T)$. Essential to our investigation is then examining how much $\mathcal{M}(T)$ differs from the image of this embedding.\\
\indent Results from \cite{BCMCG} indicate the group $\mathcal{M}(T)$ is quite complicated when $(X,T)$ is an irreducible shift of finite type; in particular, it is a countable group (see \cite[Corollary 2.3]{BCMCG}) which is not residually finite (see \cite[Theorem 6.2]{BCMCG}), and properly contains $\Aut(T) / \langle T \rangle$ (see \cite[Theorem 5.5]{BCMCG}), which is itself a large group. In \cite[Question 3.9]{BCMCG} it was asked whether, like the automorphism groups, the structure of the mapping class group of a low complexity subshift was necessarily more constrained. Two of our main results address this question directly. First we prove:
\begin{theorem}[Theorem \ref{thm:substhm1} in the main text.]\label{thm:leavesfalling}
Suppose $(X_{\xi},\sigma_{\xi})$ is a minimal subshift associated to a primitive substitution $\xi$. Then $\mathcal{M}(\sigma_{\xi})$ fits into an exact sequence
\begin{equation}
1 \to \mathcal{F} \to \mathcal{M}(\sigma_{\xi}) \to \mathbb{Z} \to 1
\end{equation}
where $\mathcal{F}$ is a finite group. If $\xi$ is of type CR (see Definition \ref{def:typeCR}), then $\mathcal{F}$ is isomorphic to $\Aut(\sigma_{\xi}) / \langle \sigma_{\xi} \rangle$.
\end{theorem}
The class of substitutions of type CR include those of Pisot type, as well as those dual to a constant length substitution.\\
\indent Key to our investigation is a representation of $\mathcal{M}(T)$ arising from its action on the ordered abelian group of coinvariants defined by
\begin{equation*}
\begin{gathered}
\mathcal{G}_{T} = C(X,\mathbb{Z}) / \langle \gamma - \gamma \circ T^{-1} \rangle\\
\mathcal{G}_{T}^{+} = \{ [\gamma] \in \mathcal{G}_{T} \mid \textnormal{ there exists non-negative } \gamma^{\prime} \in C(X,\mathbb{Z}) \textnormal{ such that } [\gamma] = [\gamma^{\prime}]\},
\end{gathered}
\end{equation*}
where $C(X,\mathbb{Z})$ denotes the group of continuous functions from $X$ to $\mathbb{Z}$. States on $\mathcal{G}_{T}$ correspond bijectively to $T$-invariant Borel probability measures on $X$, and the subgroup $\Inf \mathcal{G}_{T}$ of infinitesimals in $\mathcal{G}_{T}$ consists of the elements which vanish under every state. The (unital) ordered group $\mathcal{G}_{T}$, together with the infinitesimals, plays an important role in the classification of minimal Cantor system up to orbit equivalence (see \cite{GPS1}). Letting $P_{X}(n)$ denote the number of admissible words of length $n$ in $X$, we prove:
\begin{theorem}[Theorem \ref{thm:virtuallyabelian} in the main text.]
Let $(X,T)$ be a minimal subshift for which $\Inf \mathcal{G}_{T} = 0$ and
\begin{equation*}
\liminf\limits_{n} \frac{P_{X}(n)}{n} < \infty.
\end{equation*}
Then $\mathcal{M}(T)$ is virtually abelian.
\end{theorem}
For $(X,T)$ uniquely ergodic the condition $\Inf \mathcal{G}_{T} = 0$ is satisfied when the group of coinvariants $(\mathcal{G}_{T}, \mathcal{G}_{T}^{+})$ is totally ordered. This class includes Sturmians and certain IETs (see Remark \ref{rmk:trivialinf} for more details). We also give a complete description of the mapping class group of a Sturmian subshift in Example \ref{example:denjoy}: either the Sturmian system is conjugate to a substitution, and Theorem \ref{thm:leavesfalling} applies, or the mapping class group is trivial.\\
\indent Let us say briefly how we approach both Theorem 1.1 and 1.2. We can associate a cocycle to any self-flow equivalence $f$ of $\Sigma_{T}X$, which determines how $f$ scales leaves (i.e. path components of $\Sigma_{T}X$). There are only finitely many asymptotic leaves in the low complexity setting, so some power of $f$ must fix a leaf. The goal is to show that if $f$ acts by the identity on this leaf on average, then it is isotopic to the identity. In general, the presence of non-trivial infinitesimals provides meaningful obstructions to carrying this out. For substitution systems however, even though non-trivial infinitesimals may exist (for example, in the Thue-Morse system), the substitution map allows for an ironing out procedure (as employed in \cite{KwapiszRigidity}) that lets us overcome these obstructions.

\indent The layout of the paper is as follows. We give general definitions and background in Section 2, and then specialize to minimal systems in Section 3. Section 4 contains background on the representation of $\mathcal{M}(T)$ arising from the action on the coinvariants $\mathcal{G}_{T}$. 
In Section 5 we study $\mathcal{M}(\sigma)$ for $(X,\sigma)$ a minimal subshift associated to an aperiodic substitution and prove Theorem 1.1.
Section 6 concerns $\mathcal{M}(\sigma)$ when $(X,\sigma)$ is a subshift of linear complexity, and we prove Theorem 1.2.\\
\indent In Section 7 we construct, for any Cantor system $(X,T)$, a homomorphism from $\mathcal{M}(T)$ to the Picard group of the crossed-product $C^{*}$-algebra $C(X) \rtimes_{T} \mathbb{Z}$ associated to $(X,T)$. If $(X,T)$ is minimal, we show this homomorphism is injective, and propose studying its image in the Picard group.

\subsection{Acknowledgements}
The authors wish to thank Mike Boyle, John Franks, and Bryna Kra for several helpful suggestions and comments. This work was supported in part by the National Science Foundation grant `RTG: Analysis on manifolds' at Northwestern University.

\section{The mapping class group of a Cantor system}
By a Cantor system $(X,T)$ we will always mean a homeomorphism $T \colon X \to X$ where $X$ is a Cantor set. Given a Cantor system $(X,T)$, the \emph{suspension} $\SuspT$ is the space $X \times I / \sim$, where $(x,t) \sim (T^{n}(x),t-n)$. Elements in $\SuspT$ are equivalence classes $[(x,t)]$, where $x \in X$ and $t \in \mathbb{R}$; since every class in $\SuspT$ has a unique representative of the form $[(x,t)]$ where $0 \le t < 1$, we will often refer to points in $\SuspT$ as simply $(x,t)$ where $0 \le t < 1$. We often suppress the ordered pair notation and simply denote points by $z \in \Sigma_{T}X$, with the understanding that such a point is given by $z = (x,t)$ for some $x \in X, 0 \le t < 1$.\\
\indent The space $\SuspT$ is compact, is locally the product of a totally disconnected set with an arc, and comes with an $\mathbb{R}$-action defined by
\begin{equation}\label{eqn:thenameofawave}
\Upsilon \colon \mathbb{R} \times \SuspT \to \SuspT, \hspace{.17in} \Upsilon(s,[(x,t)]) = [(x,t+s)].
\end{equation}
We will almost exclusively use the more simple notation
\begin{equation*}
\Upsilon(s,z) = z+s.
\end{equation*}
In particular, for $k \in \mathbb{Z}$ we have $\Upsilon(k,(x,t)) = (x,t) + k = (T^{k}x,t)$.
We refer to the $\Upsilon$-orbit of a point $z$ as the \emph{leaf} containing $z$, and observe that the leaf containing $z$ coincides with the path component in $\Sigma_{T}X$ which contains $z$.\\
\indent There is a distinguished cross section to the flow that we denote by
\begin{equation*}
\Gamma = \big\{[(x,0)]  \mid x \in X \big\} \subset \Sigma_{T}X.
\end{equation*}
\indent We will at times make use of suspensions of $(X,T)$ using more general roof functions. For a positive locally constant function $r \colon X \to \mathbb{R}$ (called a roof function) we define the suspension (with roof function $r$) by
\begin{gather*}
\Sigma_{T}^{r}X  = X \times \mathbb{R} \Big/ \sim \\
(x,t) \sim (T(x),t-r(x)).
\end{gather*}
The space $\Sigma_{T}^{r}X$ also carries a flow defined analogously to that of \eqref{eqn:thenameofawave}.\\
\indent A \emph{flow equivalence} between Cantor systems $(X,T)$ and $(X^{\prime},T^{\prime})$ is a homeomorphism $f \colon \SuspT \to \Sigma_{T^{\prime}}X^{\prime}$ such that $f$ takes $\Upsilon$-orbits onto $\Upsilon$-orbits in an orientation preserving way. Two systems $(X,T), (X^{\prime},T^{\prime})$ are \emph{flow equivalent} if there exists a flow equivalence between them.\\
\indent The group of self-flow equivalences $\homeo \SuspT$ is a subgroup of the topological group of homeomorphisms of $\SuspT$, and we say $f \in \homeo \SuspT$ is isotopic to the identity if $f$ is in the path component of the identity map in $\homeo \SuspT$. Equivalently, $f \in \homeo \Sigma_{T}X$ is isotopic to the identity if and only if there exists a continuous $\eta \colon \Sigma_{T}X \to \mathbb{R}$ such that $f(z) = z+\eta(x)$ for all $z \in \Sigma_{T}X$ (see \cite[Theorem 3.1]{BCEisotopy}, together with \cite{BCEisotopycorrigendum}, for a proof of this). The collection
\begin{equation*}
\homeo _{0} \SuspT = \{ f \in \homeo \SuspT \mid f \textnormal{ is isotopic to the identity}\}
\end{equation*}
forms a normal subgroup, and we define the mapping class group of $(X,T)$ to be the quotient group
\begin{equation*}
\mathcal{M}(T) = \homeo \SuspT \Big/ \homeo_{0} \SuspT.
\end{equation*}
Thus $\mathcal{M}(T)$ consists of the group of isotopy classes of self-flow equivalences of the system $(X,T)$.
\subsection{Flow Codes}
As a consequence of the work of Parry and Sullivan in \cite{PSflowcodes}, any flow equivalence between Cantor systems is isotopic to one given by a conjugacy of return systems. We briefly outline here the parts from this theory relevant for our purposes; our presentation follows closely that of \cite{BCEisotopy}, which consolidates, and greatly expands upon, results from \cite{PSflowcodes}, and we refer the reader to \cite{BCEisotopy} for details and more background on flow equivalence and isotopy for Cantor systems.   \\
\indent A clopen subset $C \subset X$ is called a \emph{discrete cross section} if there is a continuous function $r \colon C \to \mathbb{N}$ defined by $r(x) = \min \{i \in \mathbb{N} \mid T^{i}(x) \in C\}$ for which $X = \{T^{i}(x) \mid x \in C, i \in \mathbb{N}\}$. Any discrete cross section $C \subset X$ determines a return system $(C,T_{C})$ by defining
\begin{equation*}
T_{C} \colon C \to C, \hspace{.17in} T_{C} \colon x \mapsto T^{r_{C}(x)}x.
\end{equation*}
The systems $(X,T) , (C,T_{C})$ are always flow equivalent.
Note that if $(X,T)$ is minimal, then any clopen subset is a discrete cross section, and any return system $(C,T_{C})$ is also minimal. \\
\indent For Cantor systems $(X,T), (X^{\prime},T^{\prime})$, an $(X,T),(X^{\prime},T^{\prime})$-\emph{flow code} is a triple $(\phi,C,D)$ where $C \subset X, D \subset X^{\prime}$ are discrete cross sections and $\phi \colon (C,T_{C}) \to (D,T_{D})$ is a conjugacy\footnote{When $\phi$ is a map between subshifts, $\phi$ is induced by a word block code (an analogue of a sliding block code), motivating our use of the phrase flow code; see \cite[Appendix A]{BCMCG}.}. Any $(X,T),(X^{\prime},T^{\prime})$-flow code $(\phi,C,D)$ induces a flow equivalence $S_{\phi} \colon \Sigma_{T}X \to \Sigma_{T^{\prime}}X^{\prime}$ as follows: since $C$ is a cross section, we can represent any point $z \in \Sigma_{T} X$ by $z=(x, t)$ with $x \in C$ (with $t$ possibly larger than 1). The flow code maps the segment with endpoints $(x, t)$ and $(x, t+r_C(x)) \sim (T_Cx, t)$ linearly to the segment in $\Sigma_{T'}X'$ with endpoints $(\varphi x, t)$ and $(\varphi x, t+r_D(\varphi x)) \sim (T_D (\varphi x), t)$. Since $\varphi$ is a conjugacy, the map is well-defined at endpoints. As $C$ is a discrete cross section, any leaf is a union of such segments, and the flow code defines a flow equivalence from $\Sigma_T X$ to $\Sigma_{T'}X'$. \\
\indent The following is a fundamental result in the study of flow equivalence of zero-dimensional systems. While originally due to Parry and Sullivan in \cite{PSflowcodes}, the version below is stated differently, and we refer the reader to \cite{BCEisotopy} for details.
\begin{theorem}\cite[Theorem 4.1]{BCEisotopy}\label{thm:parrysullivan}
Let $(X,T), (X^{\prime},T^{\prime})$ be Cantor systems. For any flow equivalence $f \colon \Sigma_{T}X \to \Sigma_{T^{\prime}}X^{\prime}$ there exists a flow code $(\phi,C,D)$ such that $f$ is isotopic to ${S_{\phi} \colon \Sigma_{T}X \to \Sigma_{T^{\prime}}X^{\prime}}$.
\end{theorem}
We will often refer to an $(X,X)$-flow code simply as a flow code. The following is a consequence of Theorem \ref{thm:parrysullivan}.
\begin{proposition}\cite[Corollary 2.3]{BCMCG}\label{prop:countable}
If $(X,T)$ is a subshift then $\mathcal{M}(T)$ is countable.
\end{proposition}
\begin{remark}
The subshift hypothesis in Proposition \ref{prop:countable} cannot be dropped. For odometers (see Example \ref{ex:odometers}) the group $\mathcal{M}(T)$ is uncountable.
\end{remark}

\subsection{Asymptotic leaves}
Points $z_1, z_2 \in \Sigma_{T}X$ are said to be \emph{asymptotic} (under the flow) if $\lim\limits_{t \to \infty}d(z_1+t,z_2+t) = 0$ (even though this is really \emph{forward} asymptotic, we will not be concerned with backward asymptotic points, so we suppress the term forward). We say two leaves $\ell_{1},\ell_{2}$ are asymptotic if there are points $z_{1} \in \ell_{1}, z_{2} \in \ell_{2}$ such that $z_{1}$ and $z_{2}$ are asymptotic. Note that two leaves $\ell_{1},\ell_{2}$ are asymptotic if and only if there exists points $x_{1},x_{2} \in X$ such that $(x_1,0) \in \ell_1 \cap \Gamma, (x_2,0) \in \ell_2 \cap \Gamma$, and $x_{1}, x_{2}$ are asymptotic under $T$ (so $\lim\limits_{n \to \infty}d(T^{n}x_{1},T^{n}x_{2}) = 0$). The relation defined by
\begin{equation*}
\ell_{1} \widesim{as} \ell_{2} \hspace{.11in} \textnormal{ if } \ell_{1} \textnormal{ and } \ell_{2} \textnormal{ are asymptotic}
\end{equation*}
defines an equivalence relation on leaves. We call a leaf $\ell$ asymptotic if its equivalence class $[\ell]_{as}$ under $\widesim{as}$ contains more than one element, and denote the set of equivalence classes of asymptotic leaves by $\asy$.\\
\indent Any flow equivalence must take asymptotic leaves to asymptotic leaves. As a result, there is a homomorphism
\begin{equation}\label{eqn:asymrep}
\mathcal{M}(T) \xrightarrow{\pi_{as}} P(\asy)
\end{equation}
where $P(\asy)$ denotes the permutation group on the set $\asy$.
\section{Minimal Cantor systems}
For a Cantor system $(X,T)$, any $f \in \homeo \SuspT$ permutes the set of leaves. The permutation induced by $f$ depends only on the isotopy class of $f$, giving an action of $\mathcal{M}(T)$ on the set of leaves. It can happen that for a topologically transitive subshift $(X,\sigma)$, a map $f \in \homeo \Sigma_{\sigma}X$ fixes every leaf but is not isotopic to the identity (see \cite[Example 3.3]{BCEisotopy}). The following result (which can also be obtained from \cite[Theorem 2.5]{APPsimplicity}) shows that when $(X,T)$ is minimal, this representation of $\mathcal{M}(T)$ is faithful.
\begin{proposition}\label{prop:leafaction}
Let $(X,T)$ be a minimal Cantor system and $f \in \homeo \SuspT$. The following are equivalent:
\begin{enumerate}
\item
$f$ is isotopic to the identity.
\item
$f$ is homotopic to the identity.
\item
$f$ maps every leaf to itself.
\end{enumerate}
\end{proposition}
\begin{proof}
That $(1)$ implies $(2)$ implies $(3)$ is straightforward. For $(3)$ implies $(1)$, suppose $f \in \homeo \SuspT$ maps every leaf to itself. Then there exists some $\beta \colon \SuspT \to \mathbb{R}$ satisfying
$$f(z) = z + \beta(z)$$
for all $z \in \SuspT$, and the main Theorem of \cite{KwapiszFriction} shows that, when $(X,T)$ is minimal, the map $\beta$ is continuous. One can verify that for all $t$, the map
$$f_{t}(z) = z + t \cdot \beta(z)$$
lies in $\homeo \SuspT$ (see for example \cite[Prop. 3.1]{BCEisotopy}, together with \cite{BCEisotopycorrigendum}) and hence the family $f_{t}$ provides an isotopy from $f$ to the identity.
\end{proof}

Given a discrete cross section $C \subset X$, any automorphism $\phi \in \Aut(C,T_{C})$ defines a flow code for $(X,T)$, which in turn induces a self-flow equivalence $S_{\phi}$. This gives a well-defined homomorphism
\begin{gather*}
\Psi_{C} \colon \Aut(T_{C}) \to \mathcal{M}(T)\\
\Psi_{C} \colon \varphi \mapsto [S_{\phi}].
\end{gather*}
The most fundamental case is when $C  = X$, which we denote by
$$\Psi \colon \Aut(T) \to \mathcal{M}(T).$$
The map $\Psi$ can also be defined explicitly by
$$\Psi(\varphi)([x,t]) = [\varphi(x),t].$$
\indent We call a Cantor system $(X,T)$ topologically transitive if it has a dense orbit.
\begin{proposition}\label{prop:autembedding}
Let $(X,T)$ be an aperiodic topologically transitive Cantor system and $C \subset X$ be a discrete cross section. Then $\ker \Psi_{C} = \langle T_{C} \rangle$. In particular, there is an embedding
$$\Psi \colon \Aut(T) / \langle T \rangle \hookrightarrow \mathcal{M}(T).$$
\end{proposition}
\begin{proof}
To see that $T_{C} \in \ker \Psi_{C}$, first note the map $S_{T_{C}}$ sends each leaf to itself. Moreover, the function $\beta \colon \Sigma_{T}X \to \mathbb{R}$ satisfying
\begin{equation*}
S_{T_{C}}(z) = z + \beta(z)
\end{equation*}
is continuous, from which it follows (see \cite[Prop. 3.1]{BCEisotopy}, together with \cite{BCEisotopycorrigendum}) that $S_{T_{C}}$ is isotopic to the identity.\\
\indent Conversely, suppose $\Psi_{C}(\phi)$ is isotopic to the identity. Since $(X,T)$ is topologically transitive, we may choose $x \in C$ such that the $T_{C}$-orbit of $x$ is dense in $C$. Then $\phi(x) = T^{k}(x)$ for some $k$, and since $\phi(x) \in C$, this means $\phi(x) = T_{C}^{j}(x)$ for some $j$. Thus the automorphism $T_{C}^{-j}\phi \in \Aut(T_{C})$ fixes $x$. Since $x$ has a dense $T_{C}$-orbit in $C$, this implies $\phi = T_{C}^{j}$.
\end{proof}
The following proposition, while unsurprising, is useful.
\begin{proposition}\label{prop:conjisoauto}
Let $(X,T)$ be a minimal Cantor system and let $f \in \homeo \SuspT$. If $f(z+t) = f(z)+t$ for all $z \in \Sigma_{T}X, t \in \mathbb{R}$, then $[f] \in \Im \Psi$.
\end{proposition}
\begin{proof}
The hypotheses imply there exists $0 \le s < 1$ such that $f(\Gamma) - s = \Gamma$. For any $x \in X$ there is then a unique point $y \in X$ such that $f(x,0) - s = (y,0)$, and the map $\phi_{f}(x) = y$ defines an automorphism of $(X,T)$ such that $\Psi(\phi_{f}) = [f]$ in $\mathcal{M}(T)$.
\end{proof}

\begin{example}
Example 3.9 in \cite{BLR88} constructs a minimal subshift $(Y,\sigma)$ for which $\Aut(\sigma) \cong \mathbb{Q}$ and the isomorphism takes $\sigma$ to $1 \in \mathbb{Q}$. Proposition \ref{prop:autembedding} then implies there is an embedding $\mathbb{Q} / \mathbb{Z} \longrightarrow \mathcal{M}(\sigma)$. Since any subgroup of a residually finite group must also be residually finite, and $\mathbb{Q}/\mathbb{Z}$ is not residually finite, the mapping class group $\mathcal{M}(\sigma)$ of $(Y,\sigma)$ is not residually finite.
\end{example}

Suppose $(X,T)$ is an aperiodic Cantor system. Given a self-flow equivalence $f \in \homeo \SuspT$, the map
$$\alpha_{f} \colon \SuspT \times \mathbb{R} \to \mathbb{R}$$
defined implicitly by
$$f(z+t) = f(z) + \alpha_{f}(z,t)$$
is well-defined, and satisfies the cocycle condition
\begin{equation*}
\alpha_{f}(z,t_{1} + t_{2}) = \alpha_{f}(z,t_{1}) + \alpha_{f}(z+t_{1},t_{2}).
\end{equation*}
For a flow code $(\varphi, C, D)$, the cocycle $\alpha_{S_\varphi}$ is piecewise linear, where the slopes are given by the ratios of the return times. Given $x_0 \in C$,  for $0 < t < r_{C}(x_{0})$ the slope of $\alpha_{S_\varphi}((x_0,0),t)$ is given by
\begin{equation*}
\frac{r_{D}(\phi x_{0})}{r_{C}(x_{0})}.
\end{equation*}
More generally, for
\begin{equation*}
\sum_{i=0}^{k-1}r_C(T_C^i x_0) < t < \sum_{i=0}^{k}r_C(T_C^i x_0), \hspace{.07in} k \in \mathbb{N}
\end{equation*}
the slope of $\alpha_{S_{\phi}}((x_{0},),t)$ is given by
\begin{equation}\label{eqn:linearcocycles}
\frac{r_D(T_D^k\varphi x_0)}{r_C(T_C^k x_0)}.
\end{equation}
while for $t<0$, for times
\begin{equation*}
\sum_{i=-k}^{-1} -r_C(T_C^i x_0) < t < \sum_{i=-k+1}^{-1} - r_C(T_C^i x_0), \hspace{.07in} k \in \mathbb{N}
\end{equation*}
the slope is given by
\begin{equation}
\frac{r_D(T_D^{-k}\varphi x_0)}{r_C(T_C^{-k} x_0)}.
\end{equation}


We end this section with a discussion of $\mathcal{M}(T)$ when $T$ is an odometer system. This is not a new result, and has appeared in some form or another in various works; we refer the reader to \cite{KwapiszSolenoids}, and the discussion at the end of Section 1 in \cite{KwapiszSolenoids}, for a list of some.
\begin{example}[Odometers]\label{ex:odometers}
Given an infinite sequence of prime numbers $P = \{p_{i}\}_{i=1}^{\infty}$, let $Q = \{q_{j} = \prod_{i=1}^{j} p_{i}\}_{j=1}^{\infty}$. The compact abelian group
\begin{gather*}
\mathcal{O}_{P}  = \varprojlim \{\mathbb{Z} / q_{k} \mathbb{Z},\pi_{k}\}\\
\pi_{k} \colon \mathbb{Z} /q_{k} \mathbb{Z} \to \mathbb{Z} / q_{k-1} \mathbb{Z}\\
\pi_{k} \colon a \textnormal{ mod } q_{k} \mapsto a \textnormal{ mod } q_{k-1}.
\end{gather*}
together with the translation map
\begin{gather*}
T_{P} \colon \mathcal{O}_{P} \to \mathcal{O}_{P}\\
T_{P} \colon x \mapsto x+(1,1,\ldots)
\end{gather*}
form a minimal equicontinuous Cantor system $(\mathcal{O}_{P},T_{P})$ which we refer to as the $P$-odometer. The suspension $\Sigma_{T_{P}}\mathcal{O}_{P}$ is homeomorphic to the $P$-adic solenoid
\begin{gather*}
\mathcal{S}_{P} = \varprojlim\{S^{1},w_{p_{k}}\}\\
w_{p_{k}} \colon S^{1} \to S^{1}, \hspace{.19in} w_{p_{k}} \colon z \mapsto z^{p_{k}}
\end{gather*}
We give here a presentation of $\mathcal{M}(T_{P})$ based on \cite[Theorem 1]{KwapiszSolenoids}.\\
\indent There is a short exact sequence
\begin{gather}\label{eqn:solenoidseq}
1 \to \Aut(T_{P}) / \langle T_{P} \rangle \to \mathcal{M}(T_{P}) \to \Aut(\mathcal{G}_{T_{P}},\mathcal{G}_{T_{P}}^{+}) \to 1
\end{gather}
where $\Aut(\mathcal{G}_{T_{P}},\mathcal{G}_{T_{P}}^{+})$ denotes the group of automorphisms of the abelian group
$$\mathcal{G}_{T_{P}} = \{v \in \mathbb{Q} \mid q_{k}\cdot v \in \mathbb{Z} \textnormal{ for some } k\}$$
preserving the non-negative cone
$$\mathcal{G}_{T_{P}}^{+} = \{v \in \mathbb{Q} \mid q_{k}\cdot v \in \mathbb{Z}_{+} \textnormal{ for some } k\}.$$
(The ordered group $(\mathcal{G}_{T_{P}},\mathcal{G}_{T_{P}}^{+})$ is isomorphic to the ordered group of coinvariants defined in Section \ref{sec:coinvariants}.)
When $p$ is a prime which appears infinitely often in $P$ the map
\begin{gather*}
\phi_{p} \colon \mathcal{O}_{P} \to \mathcal{O}_{P}\\
\phi_{p} \colon (x_{1},x_{2},\ldots) \longmapsto (p\cdot x_{1},p \cdot x_{2},\ldots)
\end{gather*}
is injective and $\phi_{p} \colon \mathcal{O}_{p} \to \Im \phi_{p}$ defines a flow code, giving a self-flow equivalence $S_{\phi_{p}} \in \homeo \SuspT$. (Our $S_{\phi_{p}}$ agrees with the Frobenius automorphisms defined in \cite{KwapiszSolenoids}.)

As a group $\Aut(\mathcal{G}_{T_{P}},\mathcal{G}_{T_{P}}^{+})$ is generated by the set of automorphisms
$$\{m_{p} \colon x \longmapsto p \cdot x \mid p \textnormal{ appears infinitely often in } P\}$$
and the map \begin{gather*}
\Aut(\mathcal{G}_{T_{P}},\mathcal{G}_{T_{P}}^{+}) \to \mathcal{M}(T_{P})\\
m_{p} \longmapsto S_{\phi_{p}}
\end{gather*}
defines a right-splitting map for the sequence \eqref{eqn:solenoidseq}. The group $\Aut(T_{P})$ is known to be isomorphic to $\mathcal{O}_{P}$ (see \cite[Lemma 5.9]{DDMPlowcomp}), so we have a semidirect product presentation for the mapping class group
\begin{equation*}
\mathcal{M}(T_{P}) \cong \mathcal{O}_{P} / \langle (1,1,\ldots) \rangle \rtimes \Aut(\mathcal{G}_{T_{P}},\mathcal{G}_{T_{P}}^{+}).
\end{equation*}
As an explicit example, for the set of primes $P_{2,3} = \{p_{i}\}$ defined by
$$
\begin{cases}
p_{i} = 2 & \textnormal{ if } i \textnormal{ is even}\\
p_{i} = 3 & \textnormal{ if } i \textnormal{ is odd}\\
\end{cases}
$$
we have $\Aut(\mathcal{G}_{T_{P}},\mathcal{G}_{T_{P}}^{+}) \cong \Aut(\mathbb{Z}[\frac{1}{6}],\mathbb{Z}_{+}[\frac{1}{6}]) \cong \mathbb{Z}^{2}$ and
\begin{equation*}
\mathcal{M}(T_{P_{2,3}}) \cong \mathcal{O}_{P_{2,3}} / \langle (1,1,\ldots) \rangle \rtimes \mathbb{Z}^{2}.
\end{equation*}
\end{example}

\section{The Coinvariants Representation}\label{sec:coinvariants}
Let $(X,T)$ be a Cantor system, and let $C(X,\mathbb{Z})$ denote the abelian
group of continuous integer valued functions on $X$. The map $\partial \colon
\gamma \mapsto \gamma - \gamma \circ T^{-1}$ defines a homomorphism $\partial \colon
C(X,\mathbb{Z}) \to C(X,\mathbb{Z})$, and we define the group of coinvariants
to be the abelian group
\begin{equation*}
\mathcal{G}_{T} = C(X,\mathbb{Z}) /
\textnormal{Image}(\partial).
\end{equation*}
The group $\mathcal{G}_{T}$ contains a positive cone
\begin{equation*}
\mathcal{G}_{T}^{+} = \{ [\gamma] \in \mathcal{G}_{T} \mid \textnormal{ there exists non-negative } \gamma^{\prime} \in C(X,\mathbb{Z}) \textnormal{ such that } [\gamma] = [\gamma^{\prime}]\},
\end{equation*}
and together with the distinguished order unit $[1]$ (the class of the constant function $1 \in C(X,\mathbb{Z})$), the triple $(\mathcal{G}_{T},\mathcal{G}_{T}^{+},[1])$ becomes a unital preordered group (in the sense of \cite[Section 1.2]{BHOrdered}). When $(X,T)$ is minimal, $\mathcal{G}_{T}$ is a simple dimension group (see \cite{GPS1}). By an isomorphism of ordered groups we mean an isomorphism taking the positive cone onto the positive cone; if the isomorphism in addition takes the distinguished order unit to the distinguished order unit, we say it is an isomorphism of unital ordered groups.   \\
\indent The pair $(\mathcal{G}_{T},\mathcal{G}_{T}^{+})$ is an invariant of flow equivalence: if $(X,T)$ and $(X^{\prime},T^{\prime})$ are flow equivalent, then $(\mathcal{G}_{T},\mathcal{G}_{T}^{+})$ and $(\mathcal{G}_{T^{\prime}},\mathcal{G}_{T^{\prime}}^{+})$ are isomorphic. A proof of this can be found in \cite[Theorem 1.5]{BHOrdered}. Note that the triple $(\mathcal{G}_{T},\mathcal{G}_{T}^{+},[1])$ is not in general an invariant of flow equivalence; in general, a flow equivalence may not preserve the distinguished order unit. \\
\indent Let $\pi^{1}(Y)$ denote the group $[Y,S^{1}]$ of homotopy classes of maps from $Y$ to $S^{1}$ (the group structure being given by pointwise multiplication $[f]+[g] = [fg]$). For a compact Hausdorff space $Y$, the group $\pi^{1}(Y)$ is isomorphic to the first integral \v{C}ech cohomology group $\check{H}^{1}(Y,\mathbb{Z})$ (see \cite{KrasinkiewiczMappings}).\\
\indent Given a roof function $r \colon X \to \mathbb{R}$ and map $\eta \colon \Sigma_{T}^{r}X \to S^{1}$, given $z \in \Sigma_{T}^{r}X$ there exists a continuous function $g_{z} \colon \mathbb{R} \to \mathbb{R}$ satisfying $\eta(z+t) = \eta(z)e^{2 \pi i g_{z}(t)}$ for all $t \in \mathbb{R}$. Let $C_{+}(\Sigma_{T}^{r}X,S^{1})$ denote the set of maps $\eta \colon \Sigma_{T}^{r}X \to S^{1}$ such that for all $z \in \Sigma_{T}^{r}X$, $g_{z}$ is non-decreasing. The group $\pi^{1}(\Sigma_{T}^{r}X)$ becomes a pre-ordered group by defining the positive cone associated to the `winding order'
\begin{equation*}
\pi^{1}_{+}(\Sigma_{T}^{r}X) = \{ [\eta] \mid \eta \in C_{+}(\Sigma_{T}^{r}X,S^{1})\} \subset \pi^{1}(\Sigma_{T}^{r}X).
\end{equation*}
\indent Given $\gamma \in C(X,\mathbb{Z})$, define $\eta_{\gamma} \in C(\Sigma_{T}X,S^{1})$ by
\begin{equation*}
\eta_{\gamma}(x,t) = e^{2 \pi i t \gamma(x)}, \hspace{.07in} 0 \le t < 1.
\end{equation*}
\indent The following result relating the groups $\mathcal{G}_{T}$ and $\pi^{1}(\Sigma_{T}X)$ is classical, though the claim regarding the order structures is less so. A proof may be found in \cite[Proposition 4.5]{BHOrdered}.
\begin{proposition}\label{prop:exponentialiso}
For a Cantor system $(X,T)$ the map
\begin{equation}
\begin{gathered}
\mathcal{G}_{T} \to \pi^{1}(\Sigma_{T}X)\\
[\gamma] \longmapsto [\eta_{\gamma}]
\end{gathered}
\end{equation}
is an isomorphism of groups. When $(X,T)$ is minimal, the isomorphism respects the order structures, giving an isomorphism of ordered groups
\begin{equation*}
(\mathcal{G}_{T},\mathcal{G}_{T}^{+}) \cong (\pi^{1}(\Sigma_{T}X),\pi^{1}_{+}(\Sigma_{T}X)).
\end{equation*}
\end{proposition}
In general, the isomorphism $\mathcal{G}_{T} \to \pi^{1}(\Sigma_{T}X)$ in Proposition 4.1 may not be an order isomorphism (see \cite[Example 4.7]{BHOrdered}).\\
\indent Two maps $f,g \in \homeo(\SuspT)$ which are isotopic induce the same map on
$\pi^{1}(\SuspT)$.
It follows from Proposition \ref{prop:exponentialiso} that any $[f] \in \mathcal{M}(T)$ induces an automorphism $f^{*} \in \Aut(\mathcal{G}_{T})$; the map $f^{*}$ will also preserve the positive cone $\mathcal{G}_{T}^{*}$, and there is a well-defined homomorphism
\begin{equation}\label{eqn:dimrep}
\begin{gathered}
\pi_{T} \colon \mathcal{M}(T) \to \Aut(\mathcal{G}_{T},\mathcal{G}_{T}^{+})\\
\pi_{T} \colon [f] \longmapsto f^{*}.
\end{gathered}
\end{equation}
Note that in general, $\pi_{T}$ does not take $\mathcal{M}(T)$ to $\Aut(\mathcal{G}_{T},\mathcal{G}_{T}^{+},[1])$; the induced map $f^{*}$ may not preserve the order unit $[1]$.\\
\indent Let us briefly describe how flow codes provide a convenient way to explicitly compute the map $f^{*}$ on $\mathcal{G}_{T}$ induced by some $[f] \in \mathcal{M}(T)$. Let $C \subset X$ be a discrete cross section. Given $[\gamma] \in C(X,\mathbb{Z})$, define $\gamma_{C}(x) \in C(C,\mathbb{Z})$ by
\begin{equation}
\gamma_{C}(x) =
\begin{cases}
\sum\limits_{i=0}^{r_{C}(x)-1}\gamma(T^{i}x) & \textnormal{ if } x \in C \\
     0 & \textnormal{ if } x \not \in C. \\
\end{cases}
\end{equation}
The map $res_{C} \colon [\gamma] \mapsto [\gamma_{C}]$ induces an isomorphism $res_{C} \colon \mathcal{G}_{T} \stackrel{\cong}\longrightarrow \mathcal{G}_{T_{C}}$ (a proof of this can be extracted from the proof of Theorem 1.5 in \cite{BHOrdered}). Now given $[f] \in \mathcal{M}(T)$, let $(\phi,C,D)$ be a flow code representing $[f]$. Then the map $f^{*} \colon \mathcal{G}_{T} \to \mathcal{G}_{T}$ coincides with the composition
\begin{equation}\label{eqn:moritaiso}
\mathcal{G}_{T} \xrightarrow{res_{D}} \mathcal{G}_{T_{D}} \xrightarrow{S_\phi^{*}} \mathcal{G}_{T_{C}} \xrightarrow{res_{C}^{-1}} \mathcal{G}_{T}.
\end{equation}
and the map $S_\phi^{*} \colon \mathcal{G}_{T_{D}} \to \mathcal{G}_{T_{C}}$ is given by $S_\phi^{*}([f]) = [f \circ S_\phi]_{\mathcal{G}_{T_C}}$.\\

\begin{remark}\label{remark:softlyintime}
There are minimal Cantor systems $(X,T)$ which admit automorphisms $\phi \in \Aut(T)$ such that $\pi_{T}(\phi)$ acts non-trivially on $(\mathcal{G}_{T},\mathcal{G}_{T}^{+})$. Many such examples may be found in \cite[Section 4]{MatuiFiniteOrder}. For a particular example (coming from \cite{MatuiFiniteOrder}), for the substitution defined on $\{0,1,2,3\}$ by
\begin{equation}
\xi \colon \hspace{.11in} 0 \mapsto 012230, \hspace{.11in} 1 \mapsto 123301, \hspace{.11in} 2 \mapsto 230012, \hspace{.11in} 3 \mapsto 301123
\end{equation}
the map $\phi \colon X_{\xi} \to X_{\xi}$ given by $\phi(x_{j}) = x_{j} + 1 \textnormal{ mod } 4$ defines an automorphism, and one check directly that it acts non-trivially on the dimension group (for details, we refer the reader to \cite[Section 4]{MatuiFiniteOrder}). We also note that, by a result of Matui \cite[Theorem 3.3]{MatuiFiniteOrder}, any finite subgroup in the kernel of the composition $\pi_{T} \circ \Psi \colon \Aut(T) \to \Aut(\mathcal{G}_{T},\mathcal{G}_{T}^{+})$ must be cyclic.
\end{remark}
\begin{theorem}\label{thm:coinvariantskernel}
Suppose $(X,T)$ is a minimal Cantor system. If $[f] \in \mathcal{M}(T)$ satisfies $f^{*}[1] = [1]$, then $[f] \in \Im \Psi$. In particular, $\ker \pi_{T} \subset \Im \Psi$.
\end{theorem}
\begin{proof}
Choose a flow code $(\phi,C,D)$ representing $[f]$. We claim there exists some $\gamma \in C(X,\mathbb{Z})$ such that $r_{D}\circ \phi - r_{c} = \gamma - \gamma \circ T_{C}^{-1}$. By assumption we have $[f^{*}(1)]_{\mathcal{G}_{T}} = [1]_{\mathcal{G}_{T}}$ in $\mathcal{G}_{T}$. Using  \eqref{eqn:moritaiso}, this implies that in $\mathcal{G}_{T_{C}}$
\begin{gather*}
[r_{C}]_{\mathcal{G}_{T_{C}}} = \textnormal{res}_{C}([1]_{\mathcal{G}_{T}}) = \textnormal{res}_{C}(f^{*}([1]_{\mathcal{G}_{T}})) = [r_{D} \circ S_\phi]_{\mathcal{G}_{T_{C}}}
\end{gather*}
so there exists $\gamma \in C(C,\mathbb{Z})$ such that $r_{D} \circ S_\phi - r_{C} = \gamma - \gamma \circ T_{C}^{-1}$.
Let $x_{0} \in C$. For $t_{k} = \sum_{i=0}^{k-1}r_{C}(T_{C}^{i}x_{0})$ we have
\begin{gather*}
\alpha_{f}((x_{0},0),t_{k}) - t_{k} = \sum_{i=0}^{k-1}r_{D}(\phi T_{C}^{i}x_{0}) - t_{k} = \sum_{i=0}^{k-1}r_{D}(\phi T_{C}^{i}x_{0}) -\sum_{i=0}^{k-1}r_{C}(T_{C}^{i}x_{0})\\
= \sum_{i=0}^{k-1}\gamma(T_{C}^{i}x_{0}) = \gamma(T_{C}^{k-1}x_{0}) - \gamma(T_{C}^{-1}x_{0})
\end{gather*}
which is bounded for all $k$. Since $\alpha_{f}((x_{0},0),t)$ is piecewise linear with uniformly bounded slopes between $t_{k}$ and $t_{k+1}$, $\alpha_{f}((x_{0},0),t) - t$ is bounded for $t \ge 0$. Using minimality of the $\mathbb{R}$-action, it follows that for all $z \in \Sigma_{T}X$, $\alpha_{f}(z,t) - t$ is bounded. An application of Gottschalk-Hedlund (see e.g. \cite[Theorem C]{McCutcheon99}) then implies the cocycle $\alpha_{f}(z,t) - t$ is equal to a coboundary, and there exists $\eta \colon \Sigma_{T}X \to \mathbb{R}$ such that $\alpha_{f}(z,t) - t = \eta(z) - \eta(z+t)$ for all $z \in \Sigma_{T}X, t \in \mathbb{R}$.\\
\indent Given $z \in \SuspT$, define $\beta \colon \Sigma_{T}X \to \Sigma_{T}X$ by $\beta(z) = f(z) + \eta(z)$. The map $\beta$ is continuous, and satisfies
\begin{gather*}
\beta(z+t) = f(z+t) + \eta(z+t) = f(z) + \alpha_{f}(z,t) + t - \alpha_{f}(z,t) + \eta(z)\\
= f(z) + \eta(z) + t = \beta(z) + t
\end{gather*}
for all $z \in \Sigma_{T}X, t \in \mathbb{R}$. It is then straightforward to check that $\beta$ is both injective and surjective on leaves, and hence $\beta \in \homeo \Sigma_{T}X$. Thus $\beta$ is a conjugacy of $(\Sigma_{T}X,\mathbb{R})$ to itself, so $[\beta] \in \Im \Psi$ by Proposition \ref{prop:conjisoauto}. Since $f$ is isotopic to $\beta$, $[f] \in \Im \Psi$ as well.
\end{proof}
\begin{remark}
For any $\alpha \in \Aut(T)$ it is clear that $\Psi(\alpha)^{*}([1]) = [1]$. However, by Remark \ref{remark:softlyintime} it is not true in general that $\ker \pi_{T}$ and $\Im \Psi$ agree. Theorem \ref{thm:coinvariantskernel} shows that, instead, we have $\Im \Psi = \{[f] \in \mathcal{M}(T) \mid f^{*}[1] = [1]\}$.
\end{remark}
\begin{remark}
When $(X,\sigma)$ is a non-trivial irreducible shift of finite type, it follows from \cite[Corollary 3.3]{BCMCG} that the map $\mathcal{M}(\sigma) \to \Aut(\mathcal{G}_{\sigma},\mathcal{G}_{\sigma}^{+})$ is injective.
\end{remark}
As a consequence of Theorem \ref{thm:coinvariantskernel}, for a minimal Cantor system $(X,T)$ the group $\mathcal{M}(T)$ fits into an exact sequence
\begin{equation}\label{eqn:coinvseq1}
1 \to K \to \mathcal{M}(T) \to \Im \pi_{T} \to 1
\end{equation}
where $K = \ker \pi_{T}$ can be identified with a subgroup of $\Aut(T) / \langle T \rangle$, and $\Im \pi_{T}$ is a subgroup of $\Aut(\mathcal{G}_{T},\mathcal{G}_{T}^{+})$. We remark that from a group-theoretic point of view, \eqref{eqn:coinvseq1} does not provide much restriction on how `large' $\mathcal{M}(T)$ can be. When the rank of the abelian group $\mathcal{G}_{T}$ is finite, the group $\Aut(\mathcal{G}_{T},\mathcal{G}_{T}^{+})$ embeds into $GL(k,\mathbb{R})$, where $k=\textnormal{rank} \mathcal{G}_{T}$. While the order-preserving condition does provide restrictions on $\Aut(\mathcal{G}_{T},\mathcal{G}_{T}^{+})$, even in the finite rank case the group $\Aut(\mathcal{G}_{T},\mathcal{G}_{T}^{+})$ need not be finitely-generated, nor amenable (see Examples 6.7 and 6.9 in \cite{BLR88}).
\begin{corollary}\label{cor:imagepsic}
If $[f] \in \mathcal{M}(T)$ satisfies $f^{*}([1_{C}]) = [1_{C}]$ for some clopen $C \subset X$, then $[f] \in \Im \Psi_{C}$.
\end{corollary}
\begin{proof}
Let $h \colon \Sigma_{T}X \to \Sigma_{T_{C}}C$ be a flow equivalence, giving an induced isomorphism $h_{*} \colon \mathcal{M}(T) \to \mathcal{M}(T_{C})$. Since $res_{C}([1_{C}]) = [1] \in G_{T_{C}}$, the hypotheses imply the induced action of $h_{*}([f])$ on $\mathcal{G}_{T_{C}}$ sends $[1]$ to $[1]$. By Theorem \ref{thm:coinvariantskernel}, this implies $h_{*}([f])$ is isotopic to an automorphism of $(C,T_{C})$, and hence $[f] \in \Psi_{C}$.
\end{proof}

In what follows, by a ``box'' in $\Sigma_{T}X$ we mean a set of the form \begin{equation*}
\{C+t \mid t_{0} - \epsilon \le t \le t_{0} + \epsilon\}
\end{equation*}
for some $t_{0} \in \mathbb{R}$, $\epsilon > 0$, and $C \subset X$ clopen.
\begin{theorem}\label{thm:finiteorderfrends}
Let $(X,T)$ be a minimal Cantor system. If $[f] \in \mathcal{M}(T)$ is finite order then $[f] \in \Im \Psi_{C}$ for some clopen set $C \subset X$.
\end{theorem}
\begin{proof}
Let $[f]$ be order $n > 1$. It follows from Proposition \ref{prop:leafaction} that for some leaf $\ell$, for $0 \le i \le n-1$ the leaves $f^{i}(\ell)$ are distinct. Choose $z_{0} \in l \cap \Gamma$, and for $0 \le i \le n-1$ let $z_{i} = f^{i}(z_{0})$. Define $t_{i} = \min \{ t \ge 0 \mid z_{i} + t_{i} \in \Gamma$\}, and let $y_{i} = z_{i} + t_{i}$. Note that since the leaves $f^{i}(\ell)$ are all distinct, the arcs $L_{i} = \{z_{i} + t \mid 0 \le t \le t_{i}\}$ are distinct (some of these arcs may be points, since it could be that $z_{i} \in \Gamma$). We claim there exist boxes $B_{i}$, for $1 \le i \le n-1$, which satisfy the following:
\begin{enumerate}
\item
for each $i$, $B_{i}$ contains $L_{i}$, and hence also $z_{i}$, in its interior
\item
$B_{i} \cap B_{j} = \emptyset$ if $ i \ne j$
\item
$B_{i} \cap \Gamma$ is a clopen containing $y_{i}$.
\end{enumerate}
Given these disjoint $B_{i}$, we may define $\eta \colon \Sigma_{T}X \to \mathbb{R}$ such that $\textnormal{supp}\eta \subset \cup_{i=1}^{n-1}B_{i}$ and $\eta(z_{i}) = t_{i}$. Define $S_{\eta} \in \homeo \Sigma_{T}X$ by $S_{\eta}(z) = z + \eta(z)$, so $S_{\eta}(z_{i}) = y_{i}$ for $0 \le i \le n-1$. Let $g = S_{\eta}f S_{\eta}^{-1}$ and note that, since $S_{\eta}$ is isotopic to the identity, $g$ is isotopic to $f$. Furthermore, $g$ satisfies $g(y_{i}) = y_{i+1}$ for $0 \le i \le n-2$. We now construct a finite sequence of small boxes $D_{i}$, for $0 \le i \le n-1$, which each satisfy the following properties:
\begin{enumerate}
\item
$D_{i} \cap \Gamma = E_{i}$ is clopen
\item
$y_{i} \in E_{i}$
\item
$g(D_{i}) \cap D_{i} = \emptyset$
\item
$D_{i+1} \subset g(D_{i})$ for $0 \le i \le n-2$.
\end{enumerate}
This can be done inductively: to begin, we can choose a box $D_{0}$ such that $D_{0} \cap \Gamma = E_{0}$ is clopen, $E_{0}$ contains $y_{0}$, and $g(D_{0}) \cap D_{0} = \emptyset$. Now given $D_{i}$, choose a box $D_{i+1}$ satisfying the first three properties such that $D_{i+1} \subset g(D_{i})$, using continuity, and the fact that the $y_{i}$'s are all distinct.\\
\indent Now let $D_{F} = (g^{-1})^{n-1}(D_{n-1}) = g^{1-n}(D_{n-1}) \subset D_{0}$. Then the sets $g^{i}(D_{F})$ are disjoint for all $0 \le i \le n-1$, and hence so are the clopen sets $F_{i} = g^{i}(D_{F}) \cap \Gamma$.\\
\indent We claim
\begin{equation}\label{eqn:swiftsquirrelo}
g_{*}([1_{F_{i}}]) = [1_{F_{i+1}}], \hspace{.11in} 0 \le i \le n-2.
\end{equation}
This can be seen either by examining the proof of Proposition \ref{prop:exponentialiso} in \cite[Section 4]{BHOrdered}, or by adjusting $g$ by a small isotopy so that $g$ actually takes $F_{i}$ to $F_{i+1}$.\\
\indent Now consider the element of $\mathcal{G}_{T}$
\begin{equation}
\omega = \sum_{i=0}^{n-1}(g^{i})^{*}([1_{F_{0}}]).
\end{equation}
Since $[f]$ is order $n$, $[g]$ is order $n$, so the order of $g^{*}$ divides $n$, and $g^{*}(\omega) = \omega$. On the other hand, by \eqref{eqn:swiftsquirrelo}, and using the fact that the $F_{i}$'s are disjoint, we have
\begin{equation*}
\omega = \sum_{i=0}^{n-1}[1_{F_{i}}] = [1_{F}], \hspace{.17in} F = \bigcup_{i=1}^{n-1}F_{i}.
\end{equation*}
Thus $g^{*}$, and hence $f^{*}$, fixes $[1_{F}]$, so by Corollary \ref{cor:imagepsic}, $[f] \in \Im \Psi_{F}$.
\end{proof}

\subsection{States}\label{section:states}
We call a homomorphism $\tau \colon \mathcal{G}_{T} \to \mathbb{R}$ a \emph{positive homomorphism} if it satisfies $\tau(\mathcal{G}_{T}^{+}) \subset \mathbb{R}_{+}$. A positive homomorphism $\tau$ is called a \emph{state} if $\tau([1]) = 1$. Let $p(T)$ denote the monoid of all positive homomorphisms $\tau \colon \mathcal{G}_{T} \to \mathbb{R}$, and let $\mathcal{S}(T) \subset p(T)$ denote the set of states. Any $T$-invariant Borel probability measure $\mu$ on $X$ gives rise to a state $\tau_{\mu}$ given by $\tau_{\mu}([\gamma]) = \int_{X}\gamma d \mu$, and there is a bijection between $\mathcal{S}(T)$ and the set $m(T)$ of all $T$-invariant Borel probability measures on $X$ given by
\begin{equation}\label{eqn:measurestostates}
\begin{gathered}
m(T) \to \mathcal{S}(T)\\
\mu \longmapsto \tau_{\mu}.
\end{gathered}
\end{equation}
(see \cite[Section 1.6]{BHOrdered} for details regarding this bijection). Under the correspondence in \eqref{eqn:measurestostates} an ergodic measure $\mu$ corresponds to an extremal state $\tau_{\mu}$. If $(X,T)$ has exactly $k < \infty$ $T$-invariant ergodic probability measures $\{\mu_{i}\}_{i=1}^{k}$, then $p(T)$ may be identified with the positive orthant in $\mathbb{R}^{k}$ given by $\{v = \sum_{i=1}^{k}c_{i}\tau_{\mu_{i}} \mid c_{i} \in \mathbb{R}_{+}\}$.\\
\indent The subgroup of infinitesimals of $\mathcal{G}_{T}$ is defined to be
\begin{equation}\label{eqn:infdef}
\Inf(\mathcal{G}_{T}) = \{[\gamma] \in \mathcal{G}_{T} \mid \tau([\gamma]) = 0 \textnormal{ for all } \tau \in \mathcal{S}(T)\}.
\end{equation}
Equivalently, we have
\begin{equation*}
\Inf(\mathcal{G}_{T}) = \{[\gamma] \in \mathcal{G}_{T} \mid n[\gamma] \le [1] \textnormal{ for all } n \in \mathbb{Z}\}.
\end{equation*}
\indent An automorphism $\alpha \in \Aut(\mathcal{G}_{T},\mathcal{G}_{T}^{+})$ induces a dual map $p(\alpha) \colon p(T) \to p(T)$ given by $p(\alpha)(\tau)([\gamma]) = \tau(\alpha([\gamma]))$, so there is an action of $\mathcal{M}(T)$ on the space of positive homomorphisms, giving a representation
\begin{equation}\label{eqn:staterep}
L_{T} \colon \mathcal{M}(T) \to \Aut(p(T)).
\end{equation}
For $[f] \in \mathcal{M}(T)$, we will denote $L_{T}([f])$ by $f_{*}$.
\subsection{Uniquely ergodic systems}
Recall a system $(X,T)$ is called uniquely ergodic if it has only one $T$-invariant probability measure $\mu$. In this case there is a unique state $\tau_{\mu}$ and $p(T)$ corresponds to a ray. It follows that for any $[f] \in \mathcal{M}(T)$ there exists $\lambda_{f} \in \mathbb{R}_{>0}^{*}$ such that $f_{*}\tau_{\mu} = \lambda_{f}\tau_{\mu}$
and we may identify the map $L_{T}$ in \eqref{eqn:staterep} with
\begin{equation}\label{eqn:Rmu}
\begin{gathered}
\mathcal{M}(T) \xrightarrow{R_{\mu}}\mathbb{R}_{>0}^{*}  \\
[f] \longmapsto \tau_{\mu}(f^{*}([1]))
\end{gathered}
\end{equation}
where $R_{>0}^{*}$ denotes the group of positive reals under multiplication.
\begin{lemma}\label{lemma:ratiomeasures}
Suppose $(X,T)$ is a minimal Cantor system with a unique invariant probability measure $\mu$. If $[f] \in \mathcal{M}(T)$ and $(\phi,C,D)$ is a flow code representing $[f]$, then $R_{\mu}([f]) = \mu(C)/\mu(D)$.
\end{lemma}
\begin{proof}
Since $\phi(C) = D$, we have
\begin{equation*}
\mu(C) = \tau_{\mu}(1_{C}) = \tau_{\mu}(f^{*}(1_{D})) = \lambda_{f}\tau_{\mu}(1_{D}) = \lambda_{f}\mu(D).
\end{equation*}
\end{proof}

\indent Following \cite{DOPselfinduced}, we say a minimal Cantor system $(X,T)$ is \emph{self-induced} if there exists a non-empty proper clopen set $C \subset X$ such that $(X,T)$ and $(C,T_{C})$ are conjugate. The following result of \cite{DOPselfinduced}
classifies expansive self-induced minimal systems.
\begin{theorem}\cite[Theorem 14]{DOPselfinduced}\label{thm:selfinduced}
The system $(X, T)$ is a self-induced expansive minimal Cantor system if and only if $(X,T)$ is conjugate to a substitution subshift.
\end{theorem}

\begin{proposition}\label{prop:selfinduced}
Suppose $(X,T)$ is a minimal uniquely ergodic Cantor system. If there exists $f \in \mathcal{M}(T)$ such that $R_{\mu}([f]) \ne 1$, then $(X,T)$ is flow equivalent to a self-induced system.
\end{proposition}
\begin{proof}
If $f^{*}(1) = \lambda \ne 1$, we can assume (by passing to $f^{-1}$ instead, if necessary) that $\lambda > 1$. Choosing a flow code $(\phi,C,D)$ representing $f$, by Lemma \ref{lemma:ratiomeasures} we have $f^{*}(1) = \mu(C) / \mu(D) > 1$, so $\mu(C) > \mu(D)$. Since $(X,T)$ is uniquely ergodic, we claim there exists $C^{\prime} \subset C$ and conjugacy $\tilde{\phi} \colon (D,T_{D}) \to (C^{\prime},T_{C^{\prime}})$. Indeed, a proof of this can be found in the proof of Proposition 7 in \cite{DOPselfinduced}. The composition $\tilde{\phi}\phi$ gives the self-induction map $C \to C'$. Since $(X,T)$ is flow equivalent to $(C,T_{C})$, the result follows.
\end{proof}

\begin{corollary}\label{cor:smalldeermouse}
Let $(X,T)$ be a uniquely ergodic minimal subshift which is not flow equivalent to a substitution. If $\Inf \mathcal{G}_{T} = 0$, then $\mathcal{M}(T)$ is isomorphic to $\Aut(T) / \langle T \rangle$.
\end{corollary}
\begin{proof}
Since we are assuming $(X,T)$ is not flow equivalent to a substitution, Proposition \ref{prop:selfinduced} implies $\mathcal{M}(T) = \ker R_{\mu}$. The condition $\Inf \mathcal{G}_{T} = 0$ then implies $\mathcal{M}(T) = \ker \pi_{T}$, so by Theorem \ref{thm:coinvariantskernel} we have $\mathcal{M}(T) \subset \Im \Psi$.
\end{proof}
\begin{example}\label{example:denjoy}
Let $(X_{\beta},\sigma_{\beta})$ be a minimal Sturmian subshift associated to an irrational $0 < \beta < 1$ (see \cite[Ch 6.1]{FoggSubstitutions}, \cite{MorseHedlund1}). Then $\Inf \mathcal{G}_{\sigma_{\beta}} = 0$ (this can deduced from \cite[Theorem 5.3]{PSdenjoy}), and $\Aut(\sigma_{\beta}) \cong \mathbb{Z} = \langle \sigma_{\beta} \rangle$ (see \cite[Example 4.1]{Ollisturmian}). A classical result (see for example \cite[Theorem 1]{BEIsturmian}) is that $(X_{\beta},\sigma_{\beta})$ is conjugate to a substitution system if and only if $\beta$ is a \emph{Sturm number}: that is, $\beta$ is a quadratic irrational whose algebraic conjugate lies outside the interval $[0,1]$. Thus we have the following dichotomy:
\begin{enumerate}
\item
Suppose $\beta$ is not a Sturm number. Then Corollary \ref{cor:smalldeermouse} and the above discussion implies $\mathcal{M}(\sigma_{\beta})$ is trivial.
\item
Suppose $\beta$ is a Sturm number, so $(X_{\beta},\sigma_{\beta})$ is conjugate to a substitution subshift. Then by Theorem \ref{thm:substhm1}, $\mathcal{M}(\sigma_{\beta})$ is isomorphic to $\mathcal{F} \rtimes \mathbb{Z}$ where $\mathcal{F}$ is a finite group. However $\mathcal{F}$ must be trivial: by Theorem \ref{thm:finiteorderfrends}, an element of $\mathcal{F}$ is induced by an automorphism of a return system. But this return system can have at most one pair of asymptotic points, and hence has no finite order automorphisms (by \cite{DDMPlowcomp}). Thus in this case, we have $\mathcal{M}(\sigma_{\beta}) \cong \mathbb{Z}$.
\end{enumerate}
\end{example}
\begin{remark} In Example \ref{example:denjoy}, there are non-trivial elements of the mapping class group, namely those associated to substitutions, which fix a leaf. The existence of such elements demonstrates a key difference between $\mathcal{M}(T)$ and $\Aut(T)$, since any automorphism of a minimal Cantor system $(X,T)$ which is not a power of the shift must act freely on the $T$-orbits of $X$.
\end{remark}
\begin{remark}
Flow equivalent systems must have isomorphic mapping class groups. However systems with isomorphic mapping class groups need not be flow equivalent. For example, by Fokkink's Theorem (\cite{BWDenjoy}), for irrationals $\beta, \gamma$, the Sturmian systems $(X_{\beta},\sigma_{\beta}), (X_{\gamma},\sigma_{\gamma})$ are flow equivalent if and only if $\beta$ and $\gamma$ are in the same orbit of the action of $SL(2,\mathbb{Z})$ given by $\begin{pmatrix}a & b \\ c & d \end{pmatrix}\alpha = \frac{a + b\alpha}{c + d \alpha}$. It follows then from Example \ref{example:denjoy} that there are uncountably many Sturmian systems, each of whose mapping class group is trivial, which are pairwise not flow equivalent.
\end{remark}

\section{The mapping class group of a substitution system}
Throughout this section, $\xi$ will denote a primitive aperiodic substitution on a finite alphabet $\{1, \cdots, d\}$, and we denote by $(X_{\xi},\sigma_{\xi})$  the minimal subshift associated to $\xi$. We let $\lambda_{PF}$ denote the Perron-Frobenius eigenvalue for the incidence matrix $M_{\xi}$ associated to $\xi$, and $v^{(\ell)}_{PF} = (v_{1},v_{2},\ldots,v_{d})$ a left Perron-Frobenius eigenvector.\\
\indent The first main goal of this section is to prove the following.
\begin{theorem}\label{thm:substhm1}
Suppose $(X_{\xi},\sigma_{\xi})$ is a minimal subshift associated to a primitive substitution $\xi$. Then $\mathcal{M}(\sigma_{\xi})$ fits into an exact sequence
\begin{equation}\label{eqn:subexact1}
1 \to \mathcal{F} \to \mathcal{M}(\sigma_{\xi}) \xrightarrow{R_{\mu}} \mathbb{Z} \to 1
\end{equation}
where $\mathcal{F}$ is a finite group.
\end{theorem}

Since the group on the right in \eqref{eqn:subexact1} is $\mathbb{Z}$, the sequence splits. Choosing a splitting map $s \colon \mathbb{Z} \to \mathcal{M}(\sigma_{\xi})$ gives a presentation of $\mathcal{M}(\sigma_{\xi})$ as a semidirect product $\mathcal{M}(\sigma_{\xi}) \cong \mathcal{F} \rtimes \mathbb{Z}$, and the associated structure map $\mathbb{Z} \to \Aut(\mathcal{F})$ can be determined by examining the action of $s(1)$ on the asymptotic leaves of $\Sigma_{\sigma_{\xi}}X_{\xi}$. It also follows that $\mathcal{M}(\sigma_{\xi})$ is virtually $\mathbb{Z}$. \\
\indent Let us say a bit more about the group $\mathcal{F}$. The proof of Theorem \ref{thm:substhm1} shows that $\mathcal{F}$ always embeds as a subgroup of the permutation group on the set of asymptotic classes, which is finite. Moreover, it follows from Theorem \ref{thm:finiteorderfrends} that $\mathcal{F}$ consists of self-flow equivalences induced by automorphisms of return systems of $(X_{\xi},\sigma_{\xi})$.\\
\indent The second main goal of this section is to show that, for a large class of substitutions which we now define, we can identify $\mathcal{F}$ with the group $\Aut(\sigma_{\xi}) / \langle \sigma_{\xi} \rangle $. Before defining this class, we first briefly introduce the necessary notation. \\
\indent Define the roof function $r_{PF} \colon X_{\xi} \to \mathbb{R}$ by $r_{PF}(x) = v_{x_{0}}$, where $x_0$ denotes the symbol of $x$ at the $0^{\textnormal{th}}$ coordinate. We denote the suspension $\Sigma_{T_{\xi}}^{r_{PF}}X_{\xi}$ by $\Omega_{\xi}$; the space $\Omega_{\xi}$ is also known as the tiling space, or hull, associated to an aperiodic tiling  of the real line which can be constructed from the data $\xi, v_{PF}^{(l)}$. There is a homeomorphism (see \cite[Theorem 1.1]{ClarkSadunsize})
\begin{equation}\label{eqn:singsthewhale}
\begin{gathered}
h_{\Omega} \colon \Sigma_{T}X \to \Omega_{\xi}\\
h_{\Omega} \colon (x,t) \longmapsto (x,t \cdot r_{PF}(x_{0}))
\end{gathered}
\end{equation}
and the induced map $(h_{\Omega})_{*} \colon \homeo \SuspT \to \homeo \Omega_{\xi}$ takes $\susxi$ to a substitution map which we denote by $\xi_{\Omega}$. The roof function $r_{PF}$ is distinguished in that $\Omega_{\xi}$ makes $\xi_{\Omega}$ into a self-similarity, in the sense that
\begin{equation*}
\xi_{\Omega}(x+t) = \xi_{\Omega}(x) + \lambda_{PF}t
\end{equation*}
for all $x \in \Omega_{\xi}, t \in \mathbb{R}$ (see \cite{BDcomplete}). \\

For a constant $c>0$, let $r_{c}$ denote the constant roof function $r_{c}(x) = c$. For any $c >0$, there is a homeomorphism
\begin{equation}
\begin{gathered}
h_{c} \colon \Sigma_{T}X \to \Sigma_{T}^{r_{c}}X\\
h_{c} \colon (x,t) \mapsto (x,ct).
\end{gathered}
\end{equation}
\begin{definition}\label{def:typeCR}
We say a substitution $\xi$ is of type CR if there exists a constant $c>0$ (which may depend on $\xi$) and a conjugacy $h_{CR} \colon  (\Omega_{\xi},\mathbb{R}) \to (\Sigma_{T}^{r_{c}}X_{\xi},\mathbb{R})$ satisfying the following:
\begin{equation}\label{eqn:CRleaves}
\big(h_{CR}\circ h_{\Omega} \big)(l) = h_{c}(l) \textnormal{ for every leaf } l \in \Sigma_{T}X,.
\end{equation} We denote by $\mathcal{S}_{CR}$ denote the collection of substitutions of type CR.
\end{definition}

Any substitution whose left Perron-Frobenius eigenvector is a scalar multiple of $\mathbf{1}$ (by $\mathbf{1}$ we mean the vector of all 1's) is clearly of type CR (such substitutions are precisely those dual to a constant length substitution, i.e. for every letter $i \in \mathcal{A}$, \# $i$'s in $\xi(j)$ is independent of $j$). More generally, it follows from \cite[Theorem 3.1]{ClarkSadunsize} that if the incidence matrix $M_{\xi}$ satisfies
\begin{align*}
(\lambda_{PF}^{k}v_{PF}^{(l)} - c \cdot \textbf{1})M^{m} \xrightarrow{m \to \infty} 0 \hspace{.11in} \textnormal{ for some } k
\end{align*}
then $\xi$ is of type CR using the constant $c$. From this one can deduce (see \cite[Corollary 3.2]{ClarkSadunsize}) that substitutions of Pisot type (i.e. substitutions whose Perron-Frobenius eigenvalue is a Pisot number) are of type CR.
\begin{theorem}\label{thm:splitting}
Let $\xi$ be a primitive substitution and $(X_{\xi},\sigma_{\xi})$ the minimal subshift associated to $\xi$. If $\xi$ is of type CR then the sequence
\begin{equation}\label{eqn:subsplitting}
1 \to \Aut(\sigma_{\xi})/\langle \sigma_{\xi} \rangle \stackrel{\Psi}\longrightarrow \mathcal{M}(T) \xrightarrow{R_{\mu}} \mathbb{Z} \to 1
\end{equation}
is exact.
\end{theorem}
\begin{remark}
Analogous to \eqref{eqn:subexact1}, the sequence \eqref{eqn:subsplitting} also splits, and hence in the CR case $\mathcal{M}(\sigma_{\xi})$ is isomorphic to the semidirect product $\Aut(\sigma_{\xi}) / \langle \sigma_{\xi} \rangle \rtimes \mathbb{Z}$.
\end{remark}

\indent We now prepare for the proofs of Theorems \ref{thm:substhm1} and \ref{thm:splitting}. Let $r$, be a roof function and suppose $(\Sigma_{T}^{r}X,\mathbb{R})$ has a unique $\mathbb{R}$-invariant probability measure $\mu$. Given $f \in \homeo \SuspT$ there exists $\lambda_{f} \in \mathbb{R}$ such that for any $x \in \SuspT$
\begin{equation*}
\frac{1}{t}\alpha_{f}(x,t) \xrightarrow{t \to \infty} \lambda_{f}.
\end{equation*}
This can be shown using the Ergodic Theorem (see \cite[Lemma 4.1]{KwapiszRigidity} for details).\\
\indent Recall for a space $Y$ we denote by $\pi^{1}(Y)$  the group $[Y,S^{1}]$. For any class $[g] \in \pi^{1}(\Sigma_{T}^{r}X)$, we may choose a leaf-wise smooth (with respect to the leaf-wise derivative $df(z) = \lim\limits_{t \to 0}\frac{f(z+t) - f(z)}{t}$) representative $g_{s} \colon \Sigma_{T}^{r}X \to S^{1}$ of $[g]$ and define a homomorphism
\begin{equation}
\begin{gathered}
C_{\mu} \colon \pi^{1}(\Sigma_{T}^{r}X) \longrightarrow \mathbb{R}\\
C_{\mu} \colon [g] \longmapsto \int_{\Sigma_{T}^{r}X}\frac{1}{2 \pi i}g_{s}^{-1}dg_{s} \hspace{.07in} d\mu_{r}.
\end{gathered}
\end{equation}
Upon identifying $\pi^{1}(\Sigma_{T}^{r}X)$ with $\check{H}^{1}(\Sigma_{T}^{r}X,\mathbb{Z})$, the map $C_{\mu}$ is the degree one part of the Ruelle-Sullivan map (see \cite{KPRuelleSullivan}).

The relevant cases for us we notate by
\begin{gather*}
C_{\Omega} \colon \pi^{1}(\Omega_{\xi}) \to \mathbb{R}\\
C_{\Sigma_{T}X} \colon \pi^{1}(\Sigma_{T}X) \to \mathbb{R}.
\end{gather*}
and denote the kernels by
\begin{gather*}
\Inf_{\Omega} = \ker C_{\Omega}\\
\Inf_{\Sigma_{T}X} = \ker C_{\Sigma_{T}X}.
\end{gather*}
Recall $h_{\Omega} \colon \Sigma_{T}X \to \Omega_{\xi}$ denotes the homeomorphism defined in \eqref{eqn:singsthewhale}.
\begin{lemma}\label{lemma:followtheinfs}
The isomorphism $(h_{\Omega}^{*})^{-1} \colon \pi^{1}(\Sigma_{T}X) \to \pi^{1}(\Omega_{\xi})$ induced by the homeomorphism $h_{\Omega}^{-1} \colon \Omega_{\xi} \to \Sigma_{T}X$ satisfies $(h_{\Omega}^{*})^{-1}(\Inf_{\Sigma_{T}X}) \subset \Inf_{\Omega}$.
\end{lemma}
\begin{proof}
Let $a \in \Inf_{\Sigma_{T}X}$.
By \cite[Sec. 1.9]{BHOrdered}, $a$ satisfies $na \le [1]$ for all $n \in \mathbb{Z}$, or equivalently, $[1]-na \ge 0$ for all $n$. Since $(h_{\Omega}^{-1})^{*}$ is order preserving, we have $(h_{\Omega}^{-1})^{*}([1]-na) \ge 0$ for all $n$, and hence $(h_{\Omega}^{-1})^{*}([1]) - n(h_{\Omega}^{-1})^{*}(a) \ge 0$. Note that $(h_{\Omega}^{-1})_{*}([1] \ge 0$. Then $(h_{\Omega}^{-1})^{*}([1]) \ge n (h_{\Omega}^{-1})^{*}(a)$ for all $n \in \mathbb{Z}$, which implies $C_{\Omega}(h_{\Omega}^{-1})^{*}(a) = 0$, so $(h_{\Omega}^{-1})^{*}(a) \in \Inf_{\Omega}$.
\end{proof}
Recall the isomorphism $\mathcal{G}_{T} \xrightarrow{\cong} \pi^{1}(\Sigma_{T}X)$ of Proposition \ref{prop:exponentialiso} given by
\begin{equation*}
[\gamma] \longmapsto [\eta_{\gamma}], \hspace{.07in} \eta_{\gamma}(x,t) = e^{2 \pi i t \gamma(x)}.
\end{equation*}
Since $\eta_{\gamma}^{-1}d \eta_{\gamma}(x,t) = 2 \pi i \gamma(x)$, the isomorphism $\mathcal{G}_{T} \xrightarrow{\cong} \pi^{1}(\Sigma_{T}X)$ takes $\Inf (\mathcal{G}_{T})$ to $\Inf_{\Sigma_{T}X}$.\\
\indent Define $1_{\Omega} \in C(\Omega_{\xi},S^{1})$ by $1_{\Omega}(x,t) = e^{2 \pi i t}$, and note that $C_{\Omega}(1_{\Omega}) = 1$. It follows from the Ergodic Theorem that $\lambda_{f} = C_{\Sigma_{T}^{r}X}(f^{*}(1_{\Omega}))$ (see for example \cite[Sec. 7]{JulienSadundeformations}).\\

\indent For any homeomorphism $g \colon \Sigma_{T}^{r_{1}}X \to \Sigma_{T}^{r_{2}}X$ there is an induced isomorphism of topological groups which we will denote by
\begin{gather*}
g_{*} \colon \homeo \Sigma_{T}^{r_{1}}X \to \homeo \Sigma_{T}^{r_{2}}X\\
g_{*}(f) = gfg^{-1}.
\end{gather*}
Thus for $h_{\Omega} \colon \Sigma_{T}X \to \Omega_{\xi}$ we have
\begin{gather*}
(h_{\Omega}){*} \colon \homeo \Sigma_{T}X \to \homeo \Omega_{\xi}\\
(h_{\Omega})_{*}(f) = h_{\Omega}fh_{\Omega}^{-1}.
\end{gather*}
\begin{lemma}\label{lemma:thegracefulmouse}
If $[f] \in \ker R_{\mu}$, then $\lambda_{(h_{\Omega})_{*}(f)} = 1$.
\end{lemma}
\begin{proof}
Since $[f] \in \ker R_{\mu}$, for any $a \in \mathcal{G}_{T}$ we have $f^{*}(a) - a \in \Inf(\mathcal{G}_{T})$, and hence $f^{*}h_{\Omega}^{*}(1_{\Omega}) - h_{\Omega}^{*}(1_{\Omega}) \in \Inf_{\Sigma_{T}X}$. Write $f^{*}h_{\Omega}^{*}(1_{\Omega}) = h_{\Omega}^{*}(1_{\Omega}) + b$ for some $b \in \Inf_{\Sigma_{T}X}$. Then
\begin{gather*}
C_{\Omega}\Big(\big((h_{\Omega})_{*}(f)\big)^{*}(1_{\Omega})\Big) = C_{\Omega}\big((h_{\Omega}^{-1})^{*} f^{*} h_{\Omega}^{*} (1_{\Omega})\big)\\
= C_{\Omega}\big((h_{\Omega}^{-1})^{*} (h_{\Omega}^{*}(1_{\Omega}) + b)\big) = C_{\Omega}\big(1_{\Omega} + (h_{\Omega}^{-1})^{*}(b)\big)\\
= C_{\Omega}(1_{\Omega}) = 1
\end{gather*}
where we have used that $C_{\Omega}((h_{\Omega}^{-1})^{*}(b)) = 0$ by Lemma \ref{lemma:followtheinfs}. Thus \begin{equation*}
C_{\Omega}\Big(\big((h_{\Omega})_{*}(f)\big)^{*}(1_{\Omega})\Big) = 1,
\end{equation*}
and $\lambda_{(h_{\Omega})_{*}(f)} = 1$.
\end{proof}

Finally, we note that for a primitive substitution $\xi$, $\Sigma_{\sigma_{\xi}}X$ always has a finite and non-zero number of asymptotic leaves (see either \cite[V.21]{Queff} or \cite{BDHasymptotic}).\\

The following lemma will play an important role.
\begin{lemma}\label{lemma:amongferns}
Suppose $(X_{\xi},\sigma_{\xi})$ is a minimal substitution subshift and $[f] \in \mathcal{M}(\sigma_{\xi})$. If $[f] \in \ker R_{\mu}$ and $f$ fixes an asymptotic leaf, then $f$ is isotopic to the identity.
\end{lemma}
\begin{proof}
\indent Suppose $[f] \in \ker R_{\mu}$ and fixes an asymptotic leaf $\ell$. For notational reasons, let us use the notation
\begin{equation*}
(h_{\Omega})_{*}(f) = f^{\prime} \in \homeo \Omega_{\xi}.
\end{equation*}
Since the number of asymptotic leaves is finite, there exists $j \in \mathbb{N}$ such that $\xi_{\Omega}^{j}$ fixes every asymptotic leaf. Consider the sequence of homeomorphisms
\begin{equation*}
g_{m}^{\prime} = (\xi_{\Omega}^{-j})^{m}f^{\prime}(\xi_{\Omega}^{j})^{m} = \xi_{\Omega}^{-jm}f^{\prime}\xi_{\Omega}^{jm}.
\end{equation*}
By \cite[Lemma 5.4]{KwapiszRigidity}, this sequence is equicontinuous in $\homeo \Omega_{\xi}$, and by Arzela-Ascoli converges along some subsequence $m_{k}$ to some $f^{\prime}_{\infty} \in \homeo \Omega_{\xi}$. Moreover, the homeomorphism $f^{\prime}_{\infty}$ satisfies $\alpha_{f^{\prime}_{\infty}}(x,t) = \lambda_{f^{\prime}}t$ for all $x \in \Omega_{\xi}, t \in \mathbb{R}$ (see Theorem 5.3 and its proof in \cite{KwapiszRigidity}). Since $[f] \in \ker R_{\mu}$, Lemma \ref{lemma:thegracefulmouse} implies $\lambda_{f^{\prime}} = 1$ and we have
\begin{equation}\label{eqn:treesongs}
\alpha_{f_{\infty}^{\prime}}(x,t) = t \textnormal{ for all } x \in \Omega_{\xi}, t \in \mathbb{R}.
\end{equation}
Since $h_{\Omega}(\ell)$ is an asymptotic leaf in $\Omega_{\xi}$, $\xi^{j}$ fixes $h_{\Omega}(\ell)$. Then, recalling that $f$ fixes $\ell$, we have that $g_{m}^{\prime}$, and thus $f_{\infty}^{\prime}$, fixes $h_{\Omega}(\ell)$. Since the leaf $h_{\Omega}(\ell)$ is dense in $\Omega_{\xi}$, \eqref{eqn:treesongs} implies $f_{\infty}^{\prime}$ is isotopic to the identity.\\
\indent Finally, we will show $f$ is isotopic to the identity. Let $g_{k} = (h_{\Omega})_{*}^{-1}(g_{m_{k}}^{\prime})$ and define $f_{\infty} = (h_{\Omega})_{*}^{-1}((f^{\prime})_{\infty})$. Then $g_{k} \to f_{\infty}$ in $\homeo \SuspT$  and $f_{\infty}$ is isotopic to the identity. By \cite[Prop. 1.2]{APPsimplicity}, the path component of the identity in $\homeo \SuspT$ is open, so for some $k$ we have $g_{k}$ is isotopic to the identity. But this implies $\tilde{\xi}^{-jm_{k}}f\tilde{\xi}^{jm_{k}}$ is isotopic to the identity, and hence $f$ is as well.\\
\end{proof}
We now prove Theorem \ref{thm:substhm1}.
\begin{proof}[Proof of Theorem \ref{thm:substhm1}] By Lemma \ref{lemma:amongferns}, the restriction of the asymptotic leaf representation (defined in \eqref{eqn:asymrep}) to $\ker R_{\mu}$
\begin{equation*}
\pi_{as}|_{\ker R_{\mu}} \colon \ker R_{\mu} \to P(\asy)
\end{equation*}
is injective. Since $\asy$ is finite, $P(\asy)$ is finite, showing that $\mathcal{F}$ is finite.\\
\indent It remains to show that $\Im R_\mu \subset \mathbb{R}^*$ is isomorphic to $\mathbb{Z}$. Before doing this, we record the following folklore lemma. Recall $v_{pf}^{(l)} = (v_{1},v_{2},\ldots,v_{d})$ denotes a left Perron-Frobenius eigenvector for the incidence matrix $M_{\xi}$. Define $s_{pf} = \sum_{i=1}^{d}v_{i}$, let $u_{i} = \frac{v_{i}}{s_{pf}}$, and let $u_{pf}^{(\ell)} = (u_{1},u_{2},\ldots,u_{d})$ denote the unique normalized left Perron-Frobenius eigenvector.
\begin{lemma}\label{lemma:imagetrace}
Let $\xi$ be an aperiodic primitive substitution and let $\mu$ denote the unique invariant probability measure for $(X_{\xi},T_{\xi})$. Then
\begin{equation}
\Im \tau_{\mu} = \{\frac{1}{\lambda_{\xi}^{m}}\sum_{i=0}^{d}n_{i}u_{i} \mid m \ge 0, n_{i} \in \mathbb{Z}\}.
\end{equation}
\end{lemma}
Continuing the proof of Theorem \ref{thm:substhm1}, let $S_{\xi} = \{(\frac{1}{\lambda_{\xi}})^{m}\cdot s_{pf}^{n} \mid m,n \ge 0\} \subset \mathbb{Z}[\lambda_{\xi}]$, and consider the localization $S_{\xi}^{-1}\mathbb{Z}[\lambda_{\xi}]$. Note that $\Im \tau_{\mu} \subset S_{\xi}^{-1}\mathbb{Z}[\lambda_{\xi}]$ by Lemma \ref{lemma:imagetrace}, so that $\Im R_{\mu}$ is contained in the group of units of $S_{\xi}^{-1}\mathbb{Z}[\lambda_{\xi}]$. By the $S$-unit Theorem (\cite[III.3.5]{Narkiewicz1}) (and a small argument, see for example \cite{Remymse}), the group of units in $S_{\xi}^{-1}\mathbb{Z}[\lambda_{\xi}]$ is finitely-generated. It follows that $\Im R_{\mu}$ is a finitely-generated free abelian group.\\
\indent All that remains then is to show that $\Im R_{\mu}$ is rank one. To do this, we will show the following: for any $\alpha \in \rm{Image}(R_\mu)$
there exists $p, q \in \mathbb{N}$ such that $\alpha^p = \lambda_{\xi}^q$. Given this, it follows that $\Im R_{\mu} / \langle \lambda_{\xi} \rangle$ is torsion and hence rank zero, so $\Im R_{\mu}$ is rank one as desired.

Suppose $\alpha \in \rm{Image}(R_\mu)$ and let $(\varphi,C,D)$ be a flow code with $R_\mu(\varphi) = \alpha$. By considering $\varphi^{-1}$ instead if necessary, we can assume that $\alpha>1$. Lemma \ref{lemma:ratiomeasures} implies $\mu(D) < \mu(C)$, so there exists (see the proof of \cite[Prop. 7]{DOPselfinduced}) a clopen $E \subset C$ with $\mu(E) = \mu(D)$ such that $(C, T_C)$ is conjugate to $(E, T_E)$. By Proposition \ref{prop:selfinduced}, $(C,T_{C})$ is thus a self-induced system. Note that $(C, T_C)$ is also uniquely ergodic, with unique $T_{C}$-invariant measure $\bar{\mu}$ given by $\bar{\mu}(A) = \mu(A)/\mu(C)$ for any $A \subset C$.

By \cite[Theorem 14]{DOPselfinduced}, $(C, T_{C})$ must be conjugate to an aperiodic, primitive substitution $(Y,S)$, given by $\theta:\A \to \A^*$ for some alphabet $\mathcal{A}$. Let $g \colon C \to Y$ denote such a conjugacy. The proof of \cite[Theorem 14]{DOPselfinduced} shows that that for any $a \in \A$, $|\theta(a)| = r_E(x)$, where $x \in E$ satisfies $g(x)_0 = a$. The value of the return time is well-defined, regardless of the choice of $x$. The induced transformation for the system $(\theta(Y), S_{\theta(Y)})$ is given by the lengths of substitution $\theta$ (see \cite[Cor. 5.11]{Queff})
$$S_{\theta(Y)}(z) = S^{|\theta(y_0)|}(y) \text{ where } z=\theta(y), y \in Y.$$
It follows that the return times of $x$ to $E$ are precisely the return times of $z=\theta(g(x))$ to $\theta(Y)$.

Computing the Birkhoff sums
$$\dfrac{1}{N} \sum_{i=0}^{N-1} r_U(T_U^ix)=\dfrac{1}{N} \sum_{i=0}^{N-1} r_{\theta(Y)}(S_{\theta(Y)}^iz),$$
we see that $\bar{\mu}(E) = \nu(\theta(Y))$, where $\nu$ is the unique invariant probability measure on $(Y,S)$. Thus, the Perron-Frobenius eigenvalue associated to $\theta$ is given by $1/\bar{\mu}(E) = \alpha$. By construction, $(X, T)$ and $(Y,S)$ are aperiodic primitive substitution systems which are flow equivalent. Since $\Sigma_{T}X$ is homeomorphic to $\Omega_{\xi}$ and $\Sigma_{S}Y$ is homeomorphic to $\Omega_{\theta}$, it follows that $\Omega_{\xi}$ and $\Omega_{\theta}$ are homeomorphic. Then by \cite[Theorem 2.1]{BSrigidity}, there must exist $m,n \in \mathbb{N}$ such that $\xi_{\Omega}^{m}$ and $\theta_{\Omega}^{n}$ are topologically conjugate. Since the entropy of $\xi_{\Omega}$ and $\theta_{\Omega}$ are given by $\lambda, \alpha$ respectively, we have $\lambda^{m} = \alpha^{n}$.
\end{proof}
\begin{remark}
For a substitution $\xi$, the element $[\susxi] \in \mathcal{M}(T)$ may or may not map to a generator of $\mathbb{Z}$ under the $R_{\mu}$. For example, for any substitution $\xi$ and $p \in \mathbb{N}$, let $\xi^{p}$ denote the substitution obtained by composing $\xi$ with itself $p$ times. Then $\xi$ itself gives a flow code defining a map $\susxi \in \homeo \Sigma_{T_{\xi^{p}}}X_{\xi^{P}}$, and the element $[\susxi] \in \mathcal{M}(T_{\xi^{p}})$ satisfies $(R_{\mu}([\susxi])^{p} = R_{\mu}([\tilde{\xi^{p}}])$.
\end{remark}

We now prove Theorem \ref{thm:splitting}.

\begin{proof}[Proof of Theorem \ref{thm:splitting}] Let $[f] \in \ker R_{\mu}$ and consider $f^{\prime} = (h_{\Omega})_{*}(f)$ in $\homeo \Omega_{\xi}$. As in the proof of Lemma \ref{lemma:amongferns}, the sequence $g_{m_{k}}^{\prime} = \xi_{\Omega}^{-jm_{k}}f^{\prime}\xi_{\Omega}^{jm_{k}}$ converges in $\homeo \Omega_{\xi}$ to a homeomorphism $f^{\prime}_{\infty}$ for which $\alpha_{f_{\infty}^{\prime}}(x,t) = t$ for all $x \in \Omega_{\xi}, t \in \mathbb{R}$.\\
\indent Since $\xi$ is of type CR, there exists $c > 0$ and a conjugacy $h_{CR} \colon \Omega_{\xi} \to \Sigma_{T}^{r_{c}}X $ inducing an isomorphism of topological groups $(h_{CR})_{*} \colon \homeo \Omega_{\xi} \to \homeo \Sigma_{T}^{r_{c}}X  $. Let $f_{a} = (h_{CR})_{*}(f_{\infty}^{\prime})$.\\
\indent Since $h_{CR}$ is a conjugacy, $h_{CR}(z+s) = h_{CR}(z)+s$ for all $z \in \Omega_{\xi}, s \in \mathbb{R}$. Thus $f_{a}$ is isotopic to a map preserving the cross section $X \times \{0\}$ in $\Sigma_{T}^{r_{c}}X$, and an argument analogous to the one given in Proposition \ref{prop:conjisoauto} shows there exists an automorphism $\phi \in \Aut(T)$ such that $f_{a}$ is isotopic to the map $\tilde{\phi} \in \homeo \Sigma_{T}^{r_{c}}X$ defined by
\begin{gather*}
\tilde{\phi} \colon (x,s) \longmapsto (\phi(x),s).
\end{gather*}
Let $\ell$ be an asymptotic leaf in $\Sigma_{T}X$. We claim $f^{-1}\Psi(\phi)$ in $\homeo \SuspT$ fixes $\ell$. To see this, note that since $\xi_{\Omega}^{j}$ fixes every asymptotic leaf in $\Omega_{\xi}$, we have $g_{m_{k}}^{\prime}(h_{\Omega}(\ell)) = h_{\Omega}(f(\ell))$ for all $m_{k}$, and hence $f_{\infty}^{\prime}(h_{\Omega}(\ell)) = h_{\Omega}(f(\ell))$. Likewise, $f_{a}\big(h_{CR}h_{\Omega}(\ell)\big) = \big(h_{CR}h_{\Omega}\big)(f(\ell))$. By condition \eqref{eqn:CRleaves}, this implies $f_{a}(h_{c}(\ell)) = h_{c}(f(\ell))$, and since $f_{a}$ and $\tilde{\phi}$ are isotopic, we have $\tilde{\phi}(h_{c}(\ell)) = h_{c}(f(\ell))$. One can now check that $\tilde{\phi}(h_{c}(\ell)) = h_{c}(\Psi(\phi)(\ell))$, giving $f(\ell) = \Psi(\phi)(\ell)$, so $f^{-1}\Psi(\phi)$ fixes $\ell$ as desired.\\
\indent Now to finish the proof, since both $f$ and $\Psi(\phi)$ are in $\ker R_{\mu}$, $f^{-1}\Psi(\phi) \in \ker R_{\mu}$. By the above, $f^{-1}\Psi(\phi)$ also fixes an asymptotic leaf, so by Lemma \ref{lemma:amongferns}, $f^{-1}\Psi(\phi)$ is isotopic to the identity, and $[f] = [\Psi(\phi)]$ in $\mathcal{M}(T)$.
\end{proof}


\section{The mapping class group of a subshift of linear complexity}
For a subshift $(X,T)$, let $P_{X}(n)$ denote the number of admissible words in $X$ of length $n$. We consider in this section the mapping class group associated to subshifts whose complexity function grows linearly. We note that this is a more general setting than the primitive substitution subshifts of the previous section; while substitution subshifts have linear growth of complexity, the collection of subshifts of linear complexity is a significantly larger class. On the other hand, to say something we must now assume a vanishing condition on the infinitesimals. The main result is the following.
\begin{theorem}\label{thm:virtuallyabelian}
Let $(X,T)$ be a subshift satisfying
\begin{equation*}
\liminf\limits_{n} \frac{P_{X}(n)}{n} < \infty.
\end{equation*}
If $\Inf \mathcal{G}_{T} = 0$, then $\mathcal{M}(T)$ is virtually abelian.
\end{theorem}
\begin{remark}
In \cite[Theorem 3.1]{DDMPlowcomp} it was shown that when a subshift $(X,T)$ satisfies $\liminf\limits_{n} \frac{P_{X}(n)}{n} < \infty$, the group $\Aut(T) / \langle T \rangle$ is finite. The key to this result is Lemma \ref{lemma:sleepyfern} below (which is Lemma 3.2 in \cite{DDMPlowcomp}). Furthermore, the hypothesis is not trivial; see \cite[Ex 4.1]{DDMPlowcomp} for a subshift whose complexity function $P(n)$ satisfies $\liminf_n P(n)/n$ is finite but $\limsup_n P(n)/n$ is infinite.
\end{remark}
\begin{remark} \label{rmk:trivialinf}
Any uniquely ergodic Cantor minimal system $(X,T)$ for which $(\mathcal{G}_{T},\mathcal{G}_{T}^{+})$ is totally ordered satisfies $\Inf \mathcal{G}_{T} = 0$. In particular, Theorem \ref{thm:virtuallyabelian} applies to the collection of symbolic interval exchange transformations whose coinvariants are totally ordered. In \cite[Theorem 1]{Nikolaev1} it is shown that if $(H,H^{+},[u])$ is any simple and totally ordered dimension group which is free abelian of finite rank, then there exists an interval exchange transformation $(Y_{H},S_{H})$ whose natural symbolic cover has coinvariants which are isomorphic to $(H,H^{+},[u])$.
\end{remark}

Before beginning the proof of Theorem \ref{thm:virtuallyabelian}, let us make some comments about our approach. To summarize, there are two key finiteness results for subshifts with linear complexity which we make use of: namely, finitely many asymptotic leaves (Lemma \ref{lemma:sleepyfern} below) , and finitely many ergodic measures (Theorem \ref{thm:finergodic} below). Any $[f] \in \mathcal{M}(T)$ then induces some permutation on the finite set of asymptotic leaves and the finite set of rays in the state space corresponding to extremal traces (at the end of the section, we give an example of an automorphism of a minimal system permuting two ergodic measures, showing this action can be non-trivial). What we then seek is a lemma analogous to Lemma \ref{lemma:amongferns}.\\
\indent However, the presence of non-trivial infinitesimals provides meaningful obstructions to a result like Lemma \ref{lemma:amongferns} in general, and we must settle for Lemma \ref{lemma:thewiseturtle}. While substitution systems can have non-trivial infinitesimal subgroups (for example the Thue-Morse system), it is the substitution map (in particular, the self-similar version on the space $\Omega_{\xi}$) and the ironing out procedure (as employed in \cite{KwapiszRigidity}), that lets Lemma \ref{lemma:amongferns} avoid dealing with infinitesimals. (We remark that, while it is still possibly to manually smooth out the cocycle by repeatedly refining the domain of the flow code, this approach does not work for us; in particular, we do not invoke Lemma 6.1 of \cite{KwapiszRigidity}, as the author has confirmed with us in a personal communication that the proof of the Lemma there is incomplete.)\\

\indent We first establish some preliminaries.
\begin{theorem}[\cite{Boshernitzan1}]\label{thm:finergodic}
Suppose $(X,T)$ is a subshift and there exists $k \in \mathbb{N}$ satisfying
$$\liminf\limits_{n} \frac{P_{X}(n)}{n} < k.$$
Then $(X,T)$ has at most $k-1$ ergodic non-atomic probability measures.
\end{theorem}
If $(X,T)$ is a Cantor system, by an asymptotic orbit we mean the orbit of a point $x$ for which there exists some $y \ne x$ such that $x$ and $y$ are asymptotic in $(X,T)$. Asymptotic orbits in $(X,T)$ correspond bijectively to asymptotic leaves in $\Sigma_{T}X$.
\begin{lemma}[Lemma 3.2 in \cite{DDMPlowcomp}]\label{lemma:sleepyfern}
If $(X,T)$ is a subshift for which
$$\liminf\limits_{n} \frac{P_{X}(n)}{n} < \infty$$
then the number of asymptotic orbits in $(X,T)$, and hence the number of asymptotic leaves in $\Sigma_{T}X$, is finite.
\end{lemma}
Recall $m(T)$ denotes the space of $T$-invariant finite Borel measures. By Theorem \ref{thm:finergodic} we may assume $(X,T)$ has $K$ ergodic probability measures  which we denote by $\{\mu_{i}\}_{i=1}^{K}$.\\
\indent Recall from Section \ref{section:states} there is a representation of $\mathcal{M}(T)$ on the space of positive homomorphisms
\begin{gather}
L_{T} \colon \mathcal{M}(T) \to \Aut(p(T))\\
\nonumber L_{T}([f]) = f_{*}
\end{gather}
Upon choosing an isomorphism between $\mathbb{R}^{K}$ and the linear space spanned by the collection $\{\mu_{i}\}_{i=1}^{K}$, the map $L_{T}$ gives rise to a matrix representation
$$\mathcal{L}_{T} \colon \mathcal{M}(T) \to GL_{K}(\mathbb{R}).$$
Moreover, for any $[f] \in \mathcal{M}(T)$, $L_{T}([f])$ must preserve the extremal elements of $p(T)$ and hence preserves the set of positive half-lines spanned by the $\tau_{\mu_{i}}$'s. It follows that the map $\mathcal{L}_{T}$ lands in the subgroup $GL_{K}^{perm}(\mathbb{R})$ of generalized permutation matrices, i.e. matrices which have exactly one non-zero entry in each row and column. \\
\indent Let $D \subset GL_{K}^{perm}(\mathbb{R})$ denote the subgroup of diagonal matrices. The subgroup $D$ is normal, and $GL_{K}^{perm}(\mathbb{R})$ fits in to an extension
$$1 \to D \to GL_{K}^{perm}(\mathbb{R}) \to S_{K} \to 1$$
where $S_{K}$ denotes the permutation group on $K$ symbols.\\
\indent Recall $\pi_{as} \colon \mathcal{M}(T) \to P(\asy)$ denotes the map defined in \eqref{eqn:asymrep} taking $\mathcal{M}(T)$ to the permutation group on asymptotic leaves. Let $F(T)$ denote the kernel of $\pi_{as}$. When $(X,T)$ has finitely many asymptotic equivalence classes, $F(T)$ is a finite index subgroup.
\begin{lemma}\label{lemma:thewiseturtle}
Let $(X,T)$ be a minimal Cantor system such that $\Inf \mathcal{G}_{T} = 0$. Suppose $f \in \homeo \SuspT$ fixes a leaf $l$, and satisfies $f_{*}\mu = \mu$ for all ergodic measures $\mu \in m(T)$. Then $f$ is isotopic to the identity.
\end{lemma}

\begin{proof}
Since $f_{*}\mu = \mu$ for all ergodic measures $\mu$, it follows that $\tau_{\mu}(f^{*}[1]) = f^{*}([1])$ for all extremal states $\tau_{\mu}$, and hence $\tau(f^{*}([1])) = f^{*}([1])$ for all states $\tau \in \mathcal{S}(T)$. This implies $f^{*}([1]) - [1] \in \textnormal{Inf}\mathcal{G}_{T}$, and hence $f^{*}([1]) = [1]$ since we are assuming $\textnormal{Inf}\mathcal{G}_{T} = 0$. Then Theorem \ref{thm:coinvariantskernel} implies $[f] \in \Im \Psi$. Since $f$ fixes a leaf, this implies $f$ is isotopic to the identity.
\end{proof}
\begin{proof}[Proof of Theorem \ref{thm:virtuallyabelian}]
Let $N = \mathcal{L}_T^{-1}(D) \cap F(T)$. We first show that $N$ is a finite index normal subgroup of $F(T)$. Consider the map $Q$ defined to be the composition map
$$F(T)  \xrightarrow{\mathcal{L}_T} GL_{K}^{perm}(\mathbb{R}) \to S_{K}.$$
Since $D = \textnormal{ker}(GL_{K}^{perm}(\mathbb{R}) \to S_{K})$ and $N \subset \mathcal{L}_T^{-1}(D)$, we have $N \subset \textnormal{ker}Q$. By Lemma \ref{lemma:thewiseturtle}, $\mathcal{L}_T|_{F(T)}$ is injective, and it follows that if $a \in \textnormal{ker}Q$, then $\mathcal{L}_T(a) \in D$.
Thus $\textnormal{ker}Q = N$, and since $Q$ maps to a finite group, $N$ is finite index.\\
\indent Since $N$ is finite index in $F(T)$ and $F(T)$ is finite index in $\mathcal{M}(T)$, $N$ is finite index in $\mathcal{M}(T)$.\\
\indent Finally, since $\mathcal{L}_T|_{N}$ is injective and its image lies in $D$, $N$ must also be abelian.

\end{proof}

\begin{example}[Thue-Morse]
Consider the Thue-Morse substitution given by
\begin{equation*}
\mathrel{\raisebox{0.7ex}{$\xi \colon $}}
\begin{aligned}
1 \mapsto 12\\
2 \mapsto 21
\end{aligned}
\end{equation*}

By \cite{CovenSub}, $\Aut(\sigma_{\xi})$ is isomorphic to $\mathbb{Z}/2 \mathbb{Z} \times \mathbb{Z}$, where $\mathbb{Z}/2\mathbb{Z}$ is generated by the involution $\iota$ which permutes the letters and $\mathbb{Z}$ is generated by $\sigma_{\xi}$. Since $\xi$ is of type CR (see Definition \ref{def:typeCR}) by Theorem \ref{thm:splitting} we have
$$\mathcal{M}(\sigma_{\xi}) \cong \mathbb{Z}/2 \mathbb{Z} \rtimes \mathbb{Z}.$$
In this case the involution $\iota$ commutes with the substitution, so in fact
\begin{equation*}
\mathcal{M}(\sigma_{\xi}) \cong \mathbb{Z}/2 \mathbb{Z} \times \mathbb{Z}.
\end{equation*}
We note that, while $\Aut(\sigma_{\xi})$ and $\mathcal{M}(\sigma_{\xi})$ are abstractly isomorphic in this example, the generators for the $\mathbb{Z}$ components are very different:
\begin{equation*}
\begin{gathered}
\Aut(\sigma_{\xi}) \cong \mathbb{Z}/2\mathbb{Z} \times \mathbb{Z} = \langle \iota \rangle \times \langle \sigma_{\xi} \rangle\\
\mathcal{M}(\sigma_{\xi}) \cong \mathbb{Z}/2\mathbb{Z} \times \mathbb{Z} = \langle \Psi(\iota) \rangle \times \langle \tilde{\xi} \rangle.
\end{gathered}
\end{equation*}

We also remark that the involution $\Psi(\iota) \in \mathcal{M}(\sigma_{\xi})$ acts trivially on $(\mathcal{G}_\sigma, \mathcal{G}_\sigma^+)$, as can be checked directly.
\end{example}
The following example shows that $\Aut(T)$, and hence $\mathcal{M}(T)$, can permute the ergodic measures of $(X,T)$. The construction is based on techniques from \cite{KraGeneric}.
\begin{example}\label{example:automeasurepermute}
We construct a minimal subshift of linear complexity and an automorphism which permutes two ergodic measures. Define base words
$$w_0^1 = \underbrace{0 \cdots 0}_{N_1}1,\hspace{.17in}  w_1^1 = \underbrace{1 \cdots 1}_{N_1} 0.$$
Note that $w_0^1$ is the image of $w_1^1$ under the involution $0 \leftrightarrow 1$.
We define the second level words
$$w_0^2 = \underbrace{w_0^0 \cdots w_0^0}_{N_2}w_1^0, \hspace{.17in} w_1^2 = \underbrace{w_1^1 \cdots w_{1}^{1}}_{N_2}w_0^1.$$
and recursively define $(i+1)$st level words by
$$w_0^{i+1} = \underbrace{w_0^i \cdots w_0^i}_{N_{i+1}}w_1^i, \hspace{.17in}  w_1^{i+1} = \underbrace{w_1^i \cdots w_1^i}_{N_{i+1}}w_0^i.$$
At each stage, note that $w_0^i$ is image of $w_1^i$ under the involution $0 \leftrightarrow 1$.
Consider the infinite word $w = \lim_{i \to \infty} w_0^i$. By construction, $w$ is uniformly recurrent, so the subshift $(\overline{\mathcal{O}(w)},\sigma)$ is minimal. Since the involution of $w_0^i$ is contained in $w_0^{i+1}$, the involution of $w$ is in the orbit closure of $w$. Thus, the 0-block code which permutes 0 and 1 induces an automorphism $\pi \in \mathcal{O}(w)$.
The frequency of 0s which appear in $w_0^i$ is greater than $\dfrac{N_1}{N_1+1} \cdots \dfrac{N_i}{N_i+1}$. We can choose $\{N_i\}$ to grow sufficiently quickly (for example geometrically) to ensure that the frequency of 0s in each $w_0^i$ is greater than $1/2$ and that the complexity is linear. A standard construction then gives an ergodic measure $\mu_0$ such that $\mu_0[0] > 1/2$ (see \cite{KraGeneric}). Similarly, by considering $w_1^i$, we get a distinct ergodic measure $\mu_1$ with $\mu_1([1]) > 1/2$.
Since the automorphism $\pi$ maps the cylinder set $[0]$ to $[1]$, it must send $\mu_0$ to $\mu_1$.
\end{example}

\section{$\mathcal{M}(T)$ and the Picard group of $C(X) \rtimes_{T} \mathbb{Z}$} \label{sec:picard}
For a Cantor system $(X,T)$ we will consider the crossed product $C^{*}$-algebra $\mathcal{A} = C(X) \rtimes_{T} \mathbb{Z}$ (see \cite{Putnamcpalgebras} for definitions and more regarding $\mathcal{A}$). For any $C^{*}$-algebra $\mathcal{B}$, the collection of equivalence classes of $\mathcal{B}$-$\mathcal{B}$-imprimitivity bi-modules forms a group under tensor product. This group, denoted by $\textnormal{Pic}\mathcal{B}$, is called the Picard group of $\mathcal{B}$, and was introduced by Brown, Green and Rieffel in \cite{BGRmorita}. The algebra $\mathcal{B}$ itself, considered as a $\mathcal{B}$-$\mathcal{B}$ bi-module, serves as the identity element in $\textnormal{Pic}\mathcal{B}$. For example, if one considers the $C^{*}$-algebra $C(Y)$ of continuous functions on $Y$ where $Y$ is a compact Hausdorff space, then $\textnormal{Pic}C(Y)$ is isomorphic to the semidirect product of the multiplicative group of line bundles over $Y$ with the group of homeomorphisms of $Y$. We refer the reader to \cite[Section 3]{BGRmorita} for definitions and background regarding $\textnormal{Pic} \mathcal{B}$. \\
\indent The goal of this section is to construct, for a Cantor system $(X,T)$, a homomorphism $\Theta \colon \mathcal{M}(T) \to \textnormal{Pic}\mathcal{A}$. We then show that when $(X,T)$ is minimal, $\Theta$ is injective. \\
\indent The group $\textnormal{Pic}\mathcal{A}$ has been studied for crossed product algebras arising from various systems: for irrational rotations of the circle in \cite{Kodaka97}, and for uniquely ergodic Cantor minimal systems in \cite{Nawata2012} (see also \cite{NW2010}, \cite{NW2011}, \cite{Nawata2012B}). By \cite[Prop. 3.1]{BGRmorita}, there is an injective homomorphism $\Aut\mathcal{A} / \textnormal{Inn}\mathcal{A} = \textnormal{Out}\mathcal{A} \to \textnormal{Pic}\mathcal{A}$, which is often not surjective. In some cases, e.g. substitutions, the image of the map $\Theta$ constructed here will contain classes not in the image of $\textnormal{Out}\mathcal{A}$ in $\textnormal{Pic}\mathcal{A}$. \\
\indent Before beginning, we first introduce some notation. For a clopen $C \subset X$, we let $\mathcal{A}_{C}$ denote the algebra $\chi_{C}\mathcal{A}\chi_{C}$. Note that $\mathcal{A}_{C} \cong C(C) \rtimes_{r_{C}}\mathbb{Z}$ (see \cite[Prop. 3.9]{GPS1}). We denote the $\mathcal{A}$-$\mathcal{A}_{C}$ bi-module $\mathcal{A}\chi_{C}$ by $P_{C}$. Given another clopen $D$ and an isomorphism $\mathcal{A}_{D} \xrightarrow{\alpha} \mathcal{A}_{C}$, let $\mathcal{A}_{D}^{\alpha}$ be the vector space $\mathcal{A}_{D}$, and equip $\mathcal{A}^{\alpha}_D$ with the structure of an $\mathcal{A}_D$-$\mathcal{A}_C$ bi-module defined as follows: the left action is the natural multiplication, right action is given by $x \cdot a = x \alpha^{-1}(a)$, where $x \in \mathcal{A}_D, a \in \mathcal{A}_C$, and the inner products are defined by
\begin{equation*}
_{\mathcal{A}_D}\langle x, y \rangle = xy^*, \langle x,y \rangle_{\mathcal{A}_C} = \alpha(x^{*}y).
\end{equation*}
\indent Given a flow code $(\phi,C,D)$ there is an induced isomorphism $\mathcal{A}_{D} \xrightarrow{\phi^{*}} \mathcal{A}_{C}$, and we define an $\mathcal{A}$-$\mathcal{A}$-bi-module
\begin{equation*}
X_{\phi} = P_{D} \otimes \mathcal{A}_{D}^{\phi^{*}} \otimes P_{C}^{-1}.
\end{equation*}
\indent Before defining $\Theta$, we record a useful lemma.
\begin{lemma}\label{lemma:flowcoderestriction}
Let $(\phi,C,D)$ be a flow code, $E \subset C$ a discrete cross section, and let $F = \phi(E)$. Then $(\phi|_{E},E,F)$ is a flow code.
\end{lemma}
\begin{proof}
It suffices to show that $\phi T_{E}(x) = T_{F}\phi(x)$ for any $x \in E$. Since $E \subset C$, $T_{E}(x) = T_{C}^{k}(x)$ for some $k$. Then $\phi T_{E}(x) = \phi T_{C}^{k}(x) = T_{D}^{k}\phi(x)$, so it suffices to show that $r_{F}(\phi(x)) = \sum_{i=0}^{k-1}r_{D}T_{D}^{i}(\phi(x))$.\\
\indent Since $T_{C}^{k}(x) \in E$, we have $T_{D}^{k}\phi(x) = \phi T_{C}^{k}(x) \in F$, and hence $r_{F}(\phi(x)) \le \sum_{i=1}^{k-1}r_{D}(T_{D}^{i}\phi(x))$.\\
\indent For the other direction, if $T^{j}\phi(x) \in F$, then $T^{j}\phi(x) = T_{D}^{\ell}\phi(x)$ for some $\ell$, since $F \subset D$. Then $T^{j}\phi(x) = T_{D}^{\ell}\phi(x) = \phi T_{C}^{\ell}(x)$, and hence $T_{C}^{\ell}(x) \in E$, so $\ell \ge k$. It follows that $r_{F}\phi(x)  \ge \sum_{i=0}^{k-1}r_{D}T_{D}^{i}(x)$.
\end{proof}
\indent Towards defining $\Theta$, we will show that if $(\phi,C,D)$ is any flow code, then the following hold:
\begin{enumerate}
\item
If $S_{\phi}$ is isotopic to the identity, then $X_{\phi}$ is equivalent to $_{\mathcal{A}}{\mathcal{A}}_{\mathcal{A}}$.
\item
If $E \subset C$ is a discrete cross section, the restriction $(\phi|_{E},E,\phi(E))$ is a flow code, and $X_{\phi}$ is equivalent to $X_{\phi|_{E}}$.
\item
If $(\psi,B,C)$ is another flow code, then $X_{\phi} \otimes X_{\psi}$ is equivalent to $X_{\phi \psi}$.
\end{enumerate}
\emph{Proof of (1). } Suppose $(\phi,C,D)$ is a flow code for which $S_{\phi}$ is isotopic to the identity. By Proposition \ref{prop:leafaction}, there exists $\alpha \in C(C,\mathbb{Z})$ such that $\phi(x) = T^{\alpha(x)}x$ for all $x \in C$. Let $C_{i} = \alpha^{-1}(i)$, so $\{C_{i}\}$ provides a (finite) partition of $C$, and let $v = \sum_{i \in \mathbb{Z}} u^{i} \chi_{C_{i}} \in \mathcal{A}$. Note that $v$ is a partial isometry with $v^{*}v = \chi_{C}, vv^{*} = \chi_{D}$, for which the isomorphism $\phi^{*} \colon \mathcal{A}_{D} \to \mathcal{A}_{C}$ is given by $Ad_{v^{*}}(x) = v^{*}xv$. Consider the linear map
$$\Phi \colon P_{D} \otimes \mathcal{A}_{D}^{\phi^{*}} \to P_{C}$$
given by $\Phi(a \chi_{D} \otimes \chi_{D} b \chi_{D}) = a \chi_{D} b \chi_{D} v$. For any $c \in \mathcal{A}$, we have $\Phi(cv^{*}\chi_{D} \otimes \chi_{D}) = cv^{*} \chi_{D} v= c \chi_{C}$, so $\Phi$ is onto. By \cite[Remark 3.27]{RaeburnWilliamsbook}, to show $\Phi$ is an equivalence of $\mathcal{A}$-$\mathcal{A}_{C}$-bi-modules, it is enough then to show that $\Phi$ preserves the inner products, which is a straightforward calculation.\\
\emph{Proof of (2). } First note that, by Lemma \ref{lemma:flowcoderestriction}, $\phi|_{E} \colon E \to F$ is also a flow code. We want to show that there is an equivalence of bi-modules
\begin{equation}\label{eqn:bimodulesfor2}
P_{D} \otimes \mathcal{A}_{D}^{\phi^{*}} \otimes P_{C}^{-1} \cong P_{F} \otimes \mathcal{A}_{F}^{\phi|_{E}^{*}} \otimes P_{E}^{-1}.
\end{equation}
For any clopens $V \subset U$, we may consider $P_{V}$ as an $A_{U}$-$\mathcal{A}_{V}$-bi-module, and the map $a \chi_{U} \otimes b \chi_{V} \mapsto a \chi_{U} b \chi_{V}$ induces an equivalence of bi-modules $P_{U} \otimes( _{\mathcal{A}_{U}}{P_{V}}_{\mathcal{A}_{V}}) \cong P_{V}$. Applying this to $E \subset C, F \subset D$, it follows that $P_{C}^{-1} \otimes P_{E}$ is equivalent to $ _{\mathcal{A}_{C}}{P_{E}}_{\mathcal{A}_{E}}$, and $P_{D}^{-1} \otimes P_{F}$ is equivalent to $ _{\mathcal{A}_{D}}{P_{F}}_{\mathcal{A}_{F}}$. Thus to show \eqref{eqn:bimodulesfor2} it is enough to show that $\mathcal{A}_{D}^{\phi^{*}} \otimes P_{E}$ and $P_{F} \otimes A_{F}^{\phi|_{E}^{*}}$ are equivalent. Consider the map
$$\Phi \colon \mathcal{A}_{D}^{\phi^{*}} \otimes P_{E} \to P_{F} \otimes A_{F}^{\phi|_{E}^{*}}$$
$$\Phi \colon \chi_{D} a \chi_{D} \otimes \chi_{C}b\chi_{E} \mapsto \chi_{D}a\chi_{F} \otimes (\phi|_{E}^{*})^{-1}(\chi_{E}b\chi_{E}).$$
One can check that $\Phi$ is indeed a bi-module isomorphism.\\
\emph{Proof of (3). } First observe that $\mathcal{A}_{D}^{\phi^{*}} \otimes \mathcal{A}_{C}^{\psi^{*}} \cong \mathcal{A}_{D}^{\psi^{*} \phi^{*}}$; this can be checked directly (or see \cite[Section 3]{BGRmorita}). Then we have
\begin{equation*}
\begin{gathered}
X_{\phi} \otimes X_{\psi} = P_{D} \otimes \mathcal{A}_{D}^{\phi^{*}} \otimes P_{C}^{-1} \otimes P_{C} \otimes \mathcal{A}_{C}^{\psi^{*}} \otimes P_{B}^{-1} \cong P_{D} \otimes \mathcal{A}_{D}^{\phi^{*}} \otimes \mathcal{A}_{C}^{\psi^{*}} \otimes P_{B}^{-1}\\
\cong P_{D} \otimes \mathcal{A}_{D}^{\psi^{*} \phi^{*}} \otimes P_{B}^{-1} \cong P_{D} \otimes \mathcal{A}_{D}^{(\phi \psi)^{*}} \otimes P_{B}^{-1}\\
= X_{\phi \psi}.
\end{gathered}
\end{equation*}

Now consider the map
$$\Theta_{\mathcal{M}} \colon \mathcal{M}(T) \to \textnormal{Pic}\mathcal{A}$$
defined as follows: given $[f] \in \mathcal{M}(T)$, by Theorem \ref{thm:parrysullivan} there exists a flow code $(\phi,C,D)$ such that $[S_{\phi}] = [f]$, and we set $\Theta_{\mathcal{M}}([f]) = [X_{\phi}]^{-1}$ (the choice of the inverse is to make $\Theta$ a homomorphism instead of an antihomomorphism). We will show that $\Theta$ is a well-defined homomorphism using $(1), (2)$ and $(3)$ above. First, we record the following lemma from \cite{BCMCG}, which shows how to compose flow codes (up to isotopy).
\begin{lemma}\cite[Prop. A.3]{BCMCG}\label{lemma:lemmacons}
If $(\phi,C,D), (\psi,E,F)$ are a pair of flow codes, there exists clopen subsets $K \subset C, L \subset E$ and a flow code $\eta \colon \phi(K) \to L$ such that $\eta$ is isotopic to the identity and $S_{\psi} \circ S_{\phi}$ is isotopic to $S_{\psi|_{L} \circ \eta \circ \phi|_{\phi^{-1}(K)}}$.
\end{lemma}
Note that by $(3)$ and $(1)$ we have $[X_{\phi^{-1}}] = [X_{\phi}]^{-1}$. Now to show that $\Theta_{\mathcal{M}}$ is well-defined, if $(\phi_{1},C_{1},D_{1}), (\phi_{2},C_{2},D_{2})$ are flow codes for which $[S_{\phi_{1}}] = [S_{\phi_{2}}]$, then $S_{\phi_{2}^{-1}} \circ S_{\phi_{1}}$ is isotopic to $S_{\phi_{2}}^{-1} \circ S_{\phi_{1}}$ which is isotopic to the identity. Applying Lemma \ref{lemma:lemmacons} gives $\eta, K, L$ for which $S_{\phi_{2}^{-1}|_{L} \circ \eta \circ \phi_{1}|_{\phi^{-1}(K)}}$ is isotopic to the identity. Then using $(1), (2), $ and $(3)$ above, in $\textnormal{Pic}\mathcal{A}$ we have $1 = [X_{\phi_{2}^{-1}} \otimes X_{\phi_{1}}] = [X_{\phi_{2}}]^{-1}[X_{\phi_{1}}]$, so $[X_{\phi_{1}}] = [X_{\phi_{2}}]$. That $\Theta_{\mathcal{M}}$ is a homomorphism also follows similarly using $(1), (2)$ and $(3)$.
\begin{proposition}\label{prop:picinjectivity}
If $(X,T)$ is a minimal Cantor system then $\Theta_{\mathcal{M}} \colon \mathcal{M}(T) \to \textnormal{Pic}\mathcal{A}$ is injective.
\end{proposition}
The group $\textnormal{Pic}\mathcal{A}$ acts on $K_{0}(\mathcal{A})$; in general, this action may be defined using the linking algebra, as in the discussion preceding Theorem 2.5 in \cite{Rieffelirrational}. Letting $\Aut(K_{0}(\mathcal{A}))$ denote the group of order-preserving automorphisms of $K_{0}(\mathcal{A})$ (not necessarily preserving $[1]$), there is then a homomorphism $\mathcal{P} \colon \textnormal{Pic}\mathcal{A} \to \Aut(K_{0}(\mathcal{A}))$. The groups $\mathcal{G}_{T}$ and $K_{0}(\mathcal{A})$ are isomorphic (see \cite[Theorem 1.1]{Putnamcpalgebras}), and upon identifying $\mathcal{G}_{T}$ and $K_{0}(\mathcal{A})$, the composition map $\mathcal{P} \circ \Theta_{\mathcal{M}}$
\begin{equation}\label{eqn:k0mapfactor}
\begin{gathered}
\xymatrix{
\mathcal{M}(T) \ar[r]^{\Theta_{\mathcal{M}}} \ar[dr]_{\mathcal{P} \circ \Theta_{\mathcal{M}}} & \textnormal{Pic}\mathcal{A} \ar[d]^{\mathcal{P}} \\
 & \Aut(K_{0}(\mathcal{A})) \\
 }
\end{gathered}
\end{equation}
agrees with $\pi_{T}$.\\

\begin{proof}[Proof of Proposition \ref{prop:picinjectivity}]
If $[f]$ is in the kernel of $\Theta_{M}$ then it is in the kernel of $\mathcal{P} \circ \Theta_{M}$. It follows from the above discussion that $[f]$ is then in the kernel of $\pi_{T}$, so by Theorem \ref{thm:coinvariantskernel}, $[f] = \Psi(\alpha)$ for some $\alpha \in \Aut(T)$. The composition $\Aut(T) \to \mathcal{M}(T) \to \textnormal{Pic}\mathcal{A}$ coincides with the composition map $\Aut(T) \to \textnormal{Aut}\mathcal{A} \to \textnormal{Pic}A$. If $\Theta_{\mathcal{M}} \Psi (\alpha)$ is trivial in $\textnormal{Pic}\mathcal{A}$, then $\alpha$ gives rise to an inner automorphism of $\mathcal{A}$. Since $\alpha_{*}$ preserves $C(X)$ in $\mathcal{A}$, there exists some continuous function $\delta \colon X \to \mathbb{Z}$ such that for $g \in C(X)$, $g \circ \alpha = \alpha_{*}(g) = g(T^{-\delta(x)}x)$ (this follows from \cite[Lemma 5.1]{Putnamcpalgebras}; for details, we refer the reader to the discussion preceding Proposition 2.4 in \cite{GPSfullgroup}). But this implies $\alpha$ maps a $T$-orbit to itself (in fact every $T$-orbit), and hence must be in $\langle T \rangle$, giving $[f] = [\alpha] = 1$ in $\mathcal{M}(T)$.
\end{proof}

\begin{remark}
Proposition \ref{prop:picinjectivity} also holds in the case where $(X,T)$ is a (non-trivial) irreducible shift of finite type. This follows from \eqref{eqn:k0mapfactor}, and injectivity of the map $\mathcal{M}(T) \to \Aut(K_{0}(\mathcal{A}))$, which itself can be deduced using Corollary 3.3 in \cite{BCMCG}.
\end{remark}
Recall by \cite[Prop.3.1]{BGRmorita} the kernel of the map $\Aut\mathcal{A} \to \textnormal{Pic}\mathcal{A}$ is the subgroup $\textnormal{Inn}\mathcal{A}$ of inner automorphisms of $A$, giving  an injective map $\textnormal{Out}\mathcal{A} \to \textnormal{Pic}\mathcal{A}$. It follows from Theorem \ref{thm:coinvariantskernel} that
\begin{equation}
\Im \Theta_{\mathcal{M}} \cap \textnormal{Out}\mathcal{A} = \Im \large(\Psi \colon \Aut(T) / \langle T \rangle \to \textnormal{Out}\mathcal{A}\large)
\end{equation}
It would be interesting to know how large the image of $\mathcal{M(T)}$ is under the composition
\begin{equation}
\mathcal{M}(T) \xrightarrow{\Theta_{\mathcal{M}}} \textnormal{Pic}\mathcal{A} \to \textnormal{Pic}\mathcal{A} / \textnormal{Out}\mathcal{A}
\end{equation}
and how it relates to the work in \cite{Nawatafundgroup}, \cite{NWfundgroup}.

\bibliographystyle{plain}
\bibliography{sskymcgarxiv}

\end{document}